\documentclass[10pt,draftcls,onecolumn]{IEEEtran}
\usepackage{amsmath,amssymb,graphicx,epsfig,adjustbox,color,subfigure}
\usepackage{algpseudocode}

\newcommand{\bs}[1]{\boldsymbol{#1}}
\newcommand{\mc}[1]{\mathcal{#1}}
\newcommand{\tr}{^{\text{{\tiny $T$}}}}
\newcommand{\field}[1]{{\mathbb{#1}}}
\newcommand{\be}{\begin{equation}}
\newcommand{\ee}{\end{equation}}
\newcommand{\bea}{\begin{eqnarray}}
\newcommand{\eea}{\end{eqnarray}}
\newcommand{\ba}{\begin{array}}
\newcommand{\ea}{\end{array}}
\newcommand{\beas}{\begin{eqnarray*}}
\newcommand{\eeas}{\end{eqnarray*}}
\newcommand{\leftm}{\left[\begin{array}}
\newcommand{\rightm}{\end{array}\right]}

\newtheorem{thm}{\bf Theorem}[section]

\newtheorem{prop}{\it Proposition}

\newtheorem{lem}[thm]{\bf Lemma}
\newtheorem{definition}{Definition}[section]

\DeclareMathOperator*{\argmin}{arg\,min}
\title{Energy Aware Architecture for Coordinated Mobility: An Approximate Dynamic Programming Approach}
\author{Hassan Jaleel and Jeff S. Shamma\thanks{H. Jaleel and J.S. Shamma are with the King Abdullah University of Science and Technology (KAUST), Computer, Electrical and Mathematical Sciences and Engineering Division (CEMSE), Thuwal 23955--6900, Saudi Arabia. Email: hassan.jaleel@kaust.edu.sa, jeff.shamma@kaust.edu.sa. Shamma is also with the Georgia Institute of Technology, School of Electrical and Computer Engineering. Research supported by funding from
KAUST.}}
\begin{document}
\maketitle
%----------------------------
\begin{abstract}
Our goal is to design distributed coordination strategies that enable agents to achieve global performance
guarantees while minimizing the energy cost of their actions with an emphasis on feasibility for real-time implementation. As a motivating scenario that illustrates the importance of introducing energy awareness at the agent level, we consider a team of mobile nodes that are assigned the task of establishing a communication link between two base stations with minimum energy consumption. We formulate this problem as a dynamic program in which the total cost of each agent is the sum of both mobility and communication costs. To ensure that the solution is distributed and real time implementable, we propose multiple suboptimal policies based on the concepts of approximate dynamic programming. To provide performance guarantees, we compute upper bounds on the performance gap between the proposed suboptimal policies and the global optimal policy. Finally,  we discuss merits and demerits of the proposed policies and compare their performance using simulations. 

\end{abstract}
\section {Introduction} \label{sec:Intro}
%The goal of this paper is  to present a framework for multiagent systems that will allow individual agents to strike a balance between performance and energy consumption.
The goal of this paper is to present a framework for multiagent systems that will allow individual agents to strike a balance between global performance and the cost of distributedly computing the globally optimal control action. This computation cost typically comprises  the cost of communication among neighboring nodes. In the proposed framework, the focus is on designing energy aware local interaction laws for individual agents that can provide performance guarantees, are implementable in real time, and have limited communication and computation overhead. The framework is presented in the context of a motivating example in which a collection of relay nodes is deployed to establish a communication link between two base stations for a long period of time. The task is to find an energy efficient and distributed mobility strategy to move the relay nodes to optimal locations such that the total energy consumption is minimized. We will formulate this problem as a dynamic program in which the cost to be minimized is the sum of the communication and the mobility costs of all the agents. To ensure that the control policy is distributed, energy efficient, and real time implementable, we will propose multiple suboptimal policies and compare their performance with the global optimal policy based on the techniques of approximate dynamic programming \cite{Bertsekas}.  

A fundamental challenge in designing an energy aware scheme for multiagent systems is that each agent only has limited information of the system but its actions have direct impact on the global performance (see \cite{Marden} and the references therein). Another challenge is to ensure that the proposed scheme is feasible for real-time implementation. This challenge gets more complicated when the objective is to find an optimal trajectory over a given interval because each agent must account for all the possible future trajectories of its neighbors while computing its current control actions, which makes real time implementation impractical even for simple scenarios. Consequently, the control strategies either become too complex and/or require excessive communication among the agents, resulting in large energy consumption.  Therefore, an efficient energy aware coordination strategy must have the capability to strike a balance between the performance requirements of the system and the energy requirements of the control strategies that it implements to achieve this performance. 

% With the development of low cost and reliable mobile robotic platforms with sufficient on-board computation capabilities, mobility attracted significant attention for improving network performance but it introduced  an entire new set of challenges in terms of algorithm design for efficient mobility strategies. As a result, there continues to be active research in designing local interaction rules among the neighboring agents and local control laws that guarantee to achieve certain global objectives like coverage control, formation control, and consensus in unknown environments (see e.g.,  \cite{Mesbahi, Bullo}, and \cite{Marden} and the references therein).  

The importance of introducing energy awareness in multiagent systems like wireless sensor networks is widely recognized  
%and there has been extensive research on energy efficient protocols for data forwarding and routing, sensor scheduling, area coverage, network connectivity and a wide range of other applications 
(see \cite{Anastasi} and the references therein). However, when it comes to distributed energy aware mobility strategies, the existing literature is somewhat limited and there are still a lot of unanswered questions. 
In \cite{Zhu}, a distributed energy aware coverage scheme was proposed but it only considered the tradeoff between sensing and processing and did not assign any cost to mobility. In \cite{Goldenberg}, synchronous and asynchronous distributed algorithms were presented for steering relay nodes to the optimal locations for establishing a communication link between two base stations. However, no cost was assigned to mobility because of the assumption that either the batteries can be recharged or communication is for long time and mobility cost is negligible in comparison with communication cost. In \cite{Moukaddem}, a similar problem was considered under a static setting in which agents initially determine and move to their optimal locations and then the communication starts. The proposed solutions were based on heuristics and no optimality guarantees were provided. In \cite{Jaleel} and \cite{Ali} both mobility and communication costs were considered and the problem was formulated as an optimal control problem in a dynamic setting in which agents move and communicate at the same time. In both the references, only a centralized setup was considered.  
%
 %
%In \cite{Kwok}, a power constraint deployment strategy was proposed for area coverage by explicitly relating available power to the distance that an agent can travel. Power awareness was introduced in the cost function which partitioned the environment in limited-range Voronoi regions. 
%In \cite{Jaleel} and \cite{Ali}, the problem of co-optimizing mobility and communication was addressed by formulating it as an optimal control problem in continuous time, but the proposed solutions were centralized. 
%In \cite{Cao}, an optimal linear consensus problem was formulated for both continuous time and discrete time systems with single integrator dynamics, and it was proved that for both the cases, the optimal solution is centralized. 
References \cite{Trecate, Muller, Li, Dunbar, Kazeroonio}, and \cite{Bauso} investigated optimal consensus problems using model predictive control (MPC). In all of these works, different approximate solutions were proposed with guaranteed asymptotic convergence to the consensus set, but none of these works analyzed the energy consumption profile of the proposed solutions. 

In this work, we will start by formulating the problem under consideration as an infinite horizon discounted LQR problem. Using the principal of optimality, we will formulate an equivalent one stage lookahead problem with optimal cost to go as terminal cost. This terminal cost will be a quadratic function of a positive definite matrix $K$ such that $K_{ij} \neq 0$ will indicate that the terminal cost of node $i$ depends on node $j$. We will show that all the entries of $K$ are non-zero which implies that the terminal cost of each node depends on the states of all the nodes. For a system with fixed communication network, the matrix $K$ can be computed offline before the system is deployed. Therefore, the original dynamic program will be reduced to a parameter optimization problem over a network of agents for which distributed optimization algorithms exist that can guarantee global optimial solution (see e.g., \cite{Raffard} and \cite{Nedic}). 

In \cite{Raffard}, a dual decomposition based distributed optimization algorithm was presented in which dual variables were introduced to decouple the cost of each agent. However, each agent needs to communicate with all the agents whose state directly influence its cost. Since all the entries of $K$ are non-zeros, the communication network will be a complete graph and the problem will become centralized. In \cite{Nedic}, a consensus based algorithm was presented in which only neighboring nodes of a graph were required to communicate as long as the graph was connected. However, the size of the message that agents have to communicate with their neighbors directly depends on the number of agents which can be large and can result in significant energy consumption. 

\emph{In the proposed framework, we will impose the constraints of the communication network on the optimal cost to go function and will propose approximate cost to go functions such that the cost of each agent will only depend on itself and its neighboring nodes}. This will significantly reduce the size of the communication message since the number of neighboring nodes is typically small as compared to the total nodes in the network. The price of approximating the optimal cost to go will be a loss in the global performance. We will analyze this performance loss and provide upper bounds on the performance gap between the actual and optimal performances.

%In this work, we will propose multiple energy aware policies for reducing the energy consumption of the system under consideration. The proposed policies will be easily implementable and will have minimal communication and computation overhead. We will compare the performance of each scheme with the global optimal solution and provide upper bounds for the performance gap. 

The remainder of this paper is organized as follows. Section \ref{sec:prob_description} motivates the problem under investigation and presents a mathematical formulation.  Section \ref{sec:Energy-aware architectures} presents the proposed schemes along with their performance bounds, which are the main results of this paper. Section \ref{sec:simulation} provides performance comparisons of the proposed schemes based on simulations. Finally, Section \ref{sec:} concludes the paper.

\section{Problem Description}\label{sec:prob_description}
\subsection{Notation}
%In this paper, the communication topology of the network is represented by a graph $G = (V,E)$ in which the nodes correspond to vertices $V = \{1,2,\ldots,N\}$ and the communication links correspond to edges $E = \{(i,j)~ | ~\text {$i$ and $j$ can communicate}\}$.
We denote a graph by $G = (V,E)$  where $V = \{1,2,\ldots,N\}$ is the set of vertices and $E = \{(i,j)~ | ~ i,j \in $V$\}$ is the set of edges. Graph $G$ is undirected if its links are bidirectional ($(i,j)$ $\in$ $E$ iff $(j,i)$ $\in$ $E$). The neighborhood set of node $i$ is $$\mc{N}_i = \{j \in V ~|~ (i,j) \in E\},$$ and the cardinality of this set is $\vert \mc{N}_i \vert$. A graph is connected if given any pair of nodes ($i$,$j$), either $(i,j) \in E$ or there exist some intermediate nodes such that $\{(i,i_1), (i_1,i_2),\ldots,(i_m,i_j)\}\in E$. 
The degree matrix of $G$ is a diagonal matrix $\mc{D}(G)$ with the diagonal entries $\mc{D}_{ii}(G) = |\mc{N}_i|$. The adjacency matrix $\mc{A}(G)$ is
\[\mc{A}_{ij} (G)= \left\{
  \begin{array}{ll}
    1 &:  j \in \mc{N}_i\\
    0 &:  \text{otherwise}
  \end{array}
\right.
\]
The graph laplacian, $$\mc{L}(G)  = \mc{D}(G) - \mc{A}(G),$$ is a symmetric and positive semidefinite matrix with real and non-negative eigenvalues for undirected graphs. Moreover, for a connected graph, $0 = \lambda_1<\lambda_2\leq \lambda_3\leq \ldots \leq \lambda_N$.

Let $z_i(k)$ = [$x_i(k)$ $y_i(k)$]$\tr$ denotes the location of node $i$ at time $k$ in $\field{R}^2$, where $[\cdot]\tr$ is the transpose of a vector. For concise notation, the locations of all the nodes at time $k$ are stacked in vector ${\bf z}_k\in \field{R}^{2N}$, i.e., ${\bf z}_k$ = [$z_1(k)\tr$ $\ldots$ $z_N(k)\tr$]$\tr$. For $i\in \{1,2,\ldots,N\}$, ${\bf z}_{-i,k}$ denotes the vector of locations of all neighbors of $i$, i.e., ${\bf z}_{-i,k} = [z_{i_1}(k)\tr ~ z_{i_2}(k)\tr~ \ldots ~z_{i_{|\mc{N}_i|}}(k)\tr]\tr$ where $i_1, \ldots, i_{|\mc{N}_i|} \in \mc{N}_i$.  For $z \in \field{R}^2$, $\Vert z \Vert$ denotes the Euclidean norm.

For a matrix $A $ of dimensions $N \times M$, $\Vert A \Vert_F$ is the Frobenius norm, i.e., 
$$\Vert A \Vert_F = \sqrt{\sum\limits_{i=1}^N \sum\limits_{j=1}^M |A_{ij}|^2}.$$  If $A \in \field{R}^{N \times N}$ and $\lambda_1\geq \lambda_2 \geq \ldots \lambda_N$ are its $N$ eigenvalues in non-increasing order, then $\lambda_i(A) = \lambda_i$ which implies that $\lambda_i$ will be used both as a function and as a constant value. The vector $\bs{ 1}_n$ denotes the vector $[1~ 1 \ldots1]\tr \in \field{R}^n$, $I_n \in \field{R}^{n\times n}$ denotes identity matrix, and $0_{n\times m}$ denotes a matrix of dimension $n \times m$ with all entries equal to 0. Boldface letters like ${\bf z}_k$ denote collections of vectors, and their subscript represents time. 

\subsection{Definitions}
Let $A$ be an $N\times N$ matrix. Then the spectrum of $A$ is the set of all the eigenvalues of $A$ and is denoted as $\text{spec}(A)$. The spectral radius of $A$ is $\rho(A) = \max \{|\lambda_1|,|\lambda_2|, \ldots,|\lambda_N|\}$. 
\definition{A matrix $D$ is a positive matrix ($D > 0$) if all of its entries are positive. It is a non-negative matrix ($D \geq 0$) if all of its entries are non-negative. }
 
\definition{ An $N \times N$ matrix $Z$ is an $M$ matrix if it can be written as 
\begin{equation*}
Z = s I_N - D, ~~~ s>0,~ D\geq 0,~ \rho(D) \leq s,
\end{equation*}
}
If $\rho(D) < s$ then $Z$ is a non-singular $M$ matrix and if $\rho(D) = s$ then $Z$ is a singular $M$ matrix. An essential property of an $M$ matrix is that $Z_{ii} > 0$ and $Z_{ij} \leq 0$ for all $i$ and $j$. Moreover, the inverse of an $M$ matrix is a strictly positive matrix, i.e., all entries are positive \cite{Johnson}.

\emph{Gershgorin circle theorem:} If $A$ is an $N\times N$ matrix, then every eigenvalue of $A$ lies in at least one of the circles $C_1, C_2, \ldots, C_N$ where $C_i$ has its center at the diagonal entry $c_i(A) = A_{ii}$ and its radius $r_i(A) = \sum\limits_{j\neq i} |A_{ij}|$.\\
The Gershgorin circle theorem provides upper and lower bounds for the spectrum of a square matrix. For a laplacian matrix $\mc{L}$, $c_i(\mc{L}) = r_i(\mc{L})$ since each diagonal entry is equal to the sum of the absolute values of the off diagonal entries of that row. 

\subsection{System Setup}

\begin{figure}[t]
\centering
\includegraphics[trim= 0cm 0cm 0cm 0cm,clip, scale=0.65]{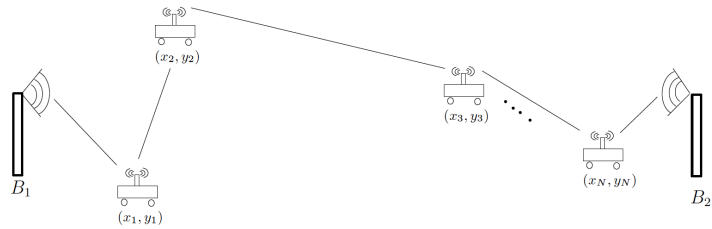}
\caption{Illustration of the communication network. $B_1$ and $B_2$ are the base stations and $N$ relay nodes are located at $(x_i,y_i)$ for all $i \in \{1,2,\ldots,N\}$.}
\label{fig:pic}
\end{figure}
  
Consider two base stations $\mc{B}_1$ and $\mc{B}_2$ separated by a distance $d$. Without loss of generality we can assume that the base station $\mc{B}_1$ is located at the origin. The objective is to establish an uninterrupted communication link between these base stations for a time interval of length $T$ with minimum energy consumption. From \cite{Goldenberg}, the power required to successfully transmit over a distance $d$ at maximum data rate is directly proportional to the square of the transmission distance $d$, i.e.,  
\begin{equation*}\label{eq:P_comm}
\text{P}_\text{comm} = (\kappa_c d^2 +b),
\end{equation*}
where $\kappa_c$ is a proportionality constant and $b$ reflects the additional power consumed by transmitter/receiver circuitry. In this work we are interested in communication power only, so assume $b = 0$. 

The energy required to establish the communication link between the two base stations will be extremely high for large values of $d$ and $T$. One solution is to use $N$  relay nodes  as shown in Fig. \ref{fig:pic}. Depending on the value of $N$ and the deployment locations of the nodes, the overall energy consumption can be significantly reduced. For minimum energy consumption, the relay nodes must be evenly spaced on the straight line between the two base stations \cite{Goldenberg}. However, we are interested in a scenario in which the nodes are randomly deployed between the base stations and they need to estimate and move to the optimal deployment locations in a decentralized manner while communicating. This redeployment towards the optimal locations will have its own cost that must be accounted for. 

Let $z_i(k)$ denotes the location of relay node $i$ at time $k$ for $i\in \{1,2,\ldots,N\}$, and $z_{N+1}(k)$ and $z_{N+2}(k)$ denote the locations of the base stations $\mc{B}_1$ and $\mc{B}_2$ respectively. Because the base stations remain stationary, $z_{N+1}(k)$ and $z_{N+2}(k)$ are constant values.  This network of the relay nodes and the base stations is represented by a graph $G_c = (V,E)$ in which the vertex set $V$ consists of $N$ relay nodes $\{1,2,\ldots,N\}$ and two base stations $\{N+1,N+2\}$. An edge exists between two vertices in $G_c$ if the corresponding nodes can communicate with each other, i.e., $\mc{N}_i$ is the set of nodes with which node $i$ can communicate. This communication network can be represented algebraically by graph laplacian $\mc{L} = \mc{L}(G_c)$ where
\[\mc{L}_{ij} = \left\{
  \begin{array}{ll}
    -1 &: i \neq j ~ \text{and} ~j \in \mc{N}_i\\
    |\mc{N}_i| &:  i=j\\
    0 &: \text{otherwise}
  \end{array}
\right.
\]

 To find the optimal locations of the relay nodes, we assume that the system satisfies the following properties.
 
\vspace{0.2in}
\emph{Assumptions:} 
\begin{enumerate}
\item The communication network is a line graph, and it remains fixed, i.e., $\mc{L}\in \field{R}^{(N+2)\times (N+2)}$ and has the following structure:
$$\mc{L} = \begin{pmatrix}
 2 & -1&0 & \cdots &\cdots& -1 & 0 \\
 -1 &2&-1&\cdots &\cdots& 0&0\\
 0&-1 &2&-1&\cdots&0&0\\
  \vdots&\vdots & \ddots  & \ddots &\ddots &  \vdots&\vdots \\
  0&0&-1&2&-1&0&0\\
0&0&0&-1&2&0&-1\\ 
-1&0&0&0&0&1&0\\ 
0&0&0&0&-1&0&1\\ 
 \end{pmatrix}$$
\item The relay nodes are mobile nodes with single integrator dynamics
\begin{equation*} %\label{eq:z_mobility}
{z}_i(k+1) = z_i(k)+u_i(k),
\end{equation*}
where $u_i (k) \in \field{R}^2$ is the input at time $k$.\vspace{0.05in} 
\end{enumerate}
Although the communication link is established initially, the overall energy consumption can be minimized if the relay nodes can somehow move to the optimal locations in an efficient manner. However, there is no leader or centralized authority that has the knowledge of the optimal locations, so the nodes need to figure it out locally based on communication with their neighbors. To ensure that the overall energy consumption is minimized, it is imperative to incorporate the cost of mobility in the system model because mobility is orders of magnitude more expensive than communication. From \cite{Mei}, mobility power consumption can be approximated to be a function of speed, i.e., 
\begin{equation*}%\label{eq:E_mob}
\text{P}_\text{mob} =  \psi( \Vert u(k) \Vert),
\end{equation*}  
where $u(k)$ is the robot's velocity at time $k$. We assume that mobility cost of a node is proportional to the square of its speed, i.e., $\psi( \Vert u(k)\Vert) = \kappa_M \Vert u(k)\Vert^2$, where $\kappa_M$ is a mobility constant. This choice of mobility model is valid for mobile nodes that use DC motors and operate at low speeds. The advantage of this model is that it will help in obtaining an analytical solution to the optimization problem.

In this paper, we will formulate this problem as a dynamic program and propose suboptimal policies that can be implemented in real time and have low energy overhead. We will also derive upper bounds on the performance gap between the proposed policies and the global optimal policy and compare their performance through simulations. 
\subsection{Problem Formulation} 
The total cost that is to be minimized is the sum of the mobility cost and the communication cost of the system. We will formulate the problem as an infinite horizon problem with a discount factor $\alpha \in (0,1)$.  \\  \\
\emph{Problem}:% \vspace{-0.2in}
\begin{align} \label{eq:Problem_infinite}
 \lim_{T \to \infty}     \min_{\{{\bf u}_0,{\bf u}_1,\ldots,{\bf u}_{T-1} \}} & \sum_{k = 0}^{T-1} \alpha^k g({\bf z}_k,{\bf u}_k)+ \alpha^T J_T({\bf z}_T)  \nonumber  \\
 \text{s.t.} ~~{\bf z}_{k+1}  &= {\bf z}_k + B{\bf u}_k, \tag{$\mc{P}$} 
\end{align}
where
\begin{equation*}
g({\bf z}_k,{\bf u}_k ) = \text{J}_{\text{comm}}({\bf z}(k))  + \text{J}_{\text{mob}}({\bf u}_k) 
\end{equation*}
%Equations for the costs
 \begin{align*}
 \text{J}_{\text{comm}}({\bf z}_k) &=\frac{1}{2}  \left(\sum_{i = 1}^N  \sum_{j \in \mc{N}_i}\kappa_C \Vert z_i(k) - z_j(k) \Vert^2\right),\\
 \text{J}_{\text{mob}}({\bf u}_k)  &= \sum_{i = 1}^N  \kappa_M \Vert  u_i(k) \Vert^2,\\
 J_T({\bf z}_T) &=\frac{1}{2}  \left(\sum_{i = 1}^N  \sum_{j \in \mc{N}_i} \kappa_C \Vert z_i(T) - z_j(T) \Vert^2\right).\\
 \end{align*}
In the system dynamics, 
$$B = \left[  \begin{array}{c}
I_{2N} \\
\hline
0_{4\times 2N}  
\end{array} \right],
$$
where $0_{4\times 2N}$ corresponds to the base stations that are stationary. The cost can be represented compactly in matrix vector notation as follows:
 \begin{align}\label{eq:cost_stage_terminal}
g({\bf z}_k,{\bf u}_k) &= {\bf z}_k\tr Q {\bf z}_k + {\bf u}_k\tr R{\bf u}_k, \nonumber\\
J_T({\bf z}_T) &= {\bf  z}_T\tr Q_T {\bf z}_T.
\end{align}
%In Eq. (\ref{eq:cost_stage_terminal}), 
Here, $R = \kappa_M I_{2N}$ is a diagonal matrix, and $Q = \kappa_C\left(\mc{L}\otimes I_2 \right)$ is a symmetric and positive semidefinite matrix of dimensions $2(N+2)\times 2(N+2)$. With linear dynamics and quadratic cost, $\mc{P}$ is an LQR problem where $u_i(k) = \mu_{i,k}({\bf z}_k)$ is the control value of agent $i$ located at $z_i(k)$ at time $k$, $\mu_{i,k}:\field{R}^2 \rightarrow \field{R}^2$ is the policy of agent $i$ at time $k$, and $\bs{\mu}_k = [\mu_{1,k}\tr~ \mu_{2,k}\tr \ldots \mu_{N,k}\tr]\tr$ is a policy vector at time $k$.

Problem $\mc{P}$ is a standard infinite horizon dynamic program with $\bs{\mu}_k$ specific stage cost $g({\bf z}_k,\bs{ \mu}_k({\bf z}_k))$ and terminal cost $J_T({\bf z}_T)$. Typically, dynamic programs suffer from the curse of dimensionality and become computationally intractable even for small size problems. However, $\mc{P}$ is an LQR problem with closed form analytical solution that involves solving the following difference Riccati equation iteratively:  
\begin{align*}\label{eq:Riccati_finite}
K_T &= Q_T \nonumber \\
%K_k &= (K_{k+1} - K_{k+1}B(B\tr K_{k+1} B + R)^{-1}\ldots  \\
%&B\tr K_{k+1}) + Q \nonumber
K_k &= (\alpha K_{k+1} - \alpha^2 K_{k+1}B(\alpha B\tr K_{k+1} B + R)^{-1} B\tr K_{k+1}) + Q,
\end{align*}
where $K_k$ is symmetric and positive semidefinite for all $k$. The optimal cost to go at stage $k$ is 
\begin{equation*}
J^*_k({\bf z}) = {\bf z}\tr K_{k} {\bf z}.
\end{equation*}
The advantage of formulating $\mc{P}$ as an infinite horizon problem is that $K_k \to K$ as $T  \to \infty$: 
\begin{equation}\label{eq:Riccati_infinite}
K = \alpha K -  \alpha^2 KB( \alpha B\tr K B + R)^{-1} B\tr K + Q.
\end{equation}
where $K$ is a symmetric and positive definite matrix. %using the results of \cite{Cao} and \cite{Alefeld}
Because of the stationarity of $K$, the optimal cost is also stationary, i.e., $J^*_k ({\bf z}) \to J^* ({\bf z}) $ as $k \to \infty$, where
\begin{equation}\label{eq:cost_optimal}
J^* ({\bf z}) = {\bf z}\tr K {\bf z}.
\end{equation}

Using the principal of optimality, $\mc{P}$ can be formulated as a one stage lookahead problem with stationary optimal cost to go as its terminal cost, i.e., 
\begin{equation}\label{eq:OneStepOptimal}
\bs{\mu}^*({\bf z}_k) = \argmin_{{\bf u}} \left[g({\bf z}_k,{\bf u}) + \alpha J^*({\bf z}_k + B{\bf u})\right]
\end{equation}
subject to the dynamics and constraints in \ref{eq:Problem_infinite}. The optimal stationary policy $\bs{\mu}^*$ has the form
\begin{equation}\label{eq:optimal_control}
\begin{split}
\bs{\mu}^*({\bf z}_k) &= -L {\bf z}_k, \\
L &= \alpha(\alpha B\tr KB + R)^{-1}B\tr K.
\end{split}
\end{equation}
The optimal policy is a simple static feedback law that can be implemented easily under normal circumstances. However, we will show that the optimal policy is centralized which is an intuitive result and makes this problem challenging since the optimal solution requires complete knowledge of the network and cannot be decentralized. 

Next we will show that $K$ is a non-singular $M$-matrix and all of its entries are non-zero, which means each agent requires state information of all the other agents to compute the optimal cost to go. 
%Since $K$ is an $M$ matrix and a positive definite matrix, so $K_{ij} \leq 0$ for all $i \in \{1,2,\ldots,2(N+2)\}$ and $i\neq j$ and $K_{ii} > -\sum_{i = 1}^{2(N+2)}K_{ij}$. 
\begin{prop}\label{prop:K_Matrix}
\emph{If $Q \in \field{R}^{2(N+2)\times 2(N+2)}$ is a symmetric $M$ matrix}, 
$B = \left[  \begin{array}{c}
I_{2N} \\
\hline
0_{4 \times 2N}  
\end{array} \right]
$, and 
$R = \kappa_M I_{2N}$, 
\emph{then the positive definite solution of the Riccati equation}
\begin{equation*}
K =  \alpha K -  \alpha^2 KB( \alpha B\tr K B + R)^{-1} B\tr K + Q
\end{equation*}
\emph{is also an $M$ matrix. Moreover, if $Q$ is a laplacian matrix of a connected graph then all the entries of $K$ are non-zero. }
\end{prop}
\begin{proof}
The proof is presented in the appendix.
\end{proof}
\vspace{0.1in}

Problems having structure similar to $\mc{P}$ have been studied in the context of optimal consensus problems for general LTI systems (see for example \cite{Trecate, Muller, Dunbar, Kazeroonio, Bauso}, and \cite{Cao}). In all of these references, consensus problem was formulated as an infinite horizon LQR problem with the objective of minimizing disagreement among the agents with minimum control effort. In \cite{Cao} it was shown that the optimal solution to this problem is centralized. Since the problem could not be solved in a decentralized manner, approximations were introduced to decouple the cost along the entire trajectory.  In \cite{Trecate, Li, Kazeroonio} and \cite{Bauso} each agent assumed that all of its neighbors remain stationary for all the future time. Based on this assumption, optimal control trajectories were computed and implemented for one time step, a standard approach in MPC, and the process was repeated. In \cite{Muller}, the same stationarity assumption was used but the update rule was considered to be sequential, i.e., at each time only one agent was allowed to compute its optimal trajectory. In \cite{Dunbar}, each agent transmitted its assumed trajectory to all of its neighbors. At the same time it received such trajectories from its neighbors. Then, each agent used these assumed trajectories of its neighbors to solve its optimization problem and computed its actual trajectory. For stability, the actual trajectory and the assumed trajectory of each agent should be close.  

Although $\mc{P}$ has the same cost as the references mentioned, the goal is significantly different. The schemes in the references are primarily designed for achieving a certain steady state performance (consensus). To ensure that the trajectories computed by agents are feasible, and converge to the consensus set, these schemes rely on extensive communication among neighboring agents. In fact, some schemes require communication of entire trajectories among the neighboring agents repeatedly at each decision time step (\cite{Muller} and \cite{Keviczky}). This communication results in considerable energy overhead that is typically not taken into account. 

\emph{Our focus is on developing a framework for minimizing this communication overhead and ensuring that the scheme is distributed and real time implementable.} The price of reducing inter-agent communication will be in terms of global performance. Therefore, we will analyze the effects on system performance and provide upper bounds on the performance gap between the proposed and optimal policies. Furthermore, the proposed framework is not limited to a particular problem. Instead, it is intended to provide energy aware local coordination algorithms that can provide performance guarantees for a class of multiagent systems with fixed communication network.

\subsection{Main Contribution}   
Based on the above discussion, it is inefficient in terms of communication energy overhead to solve problem $\mc{P}$ in its original form. Another formulation of this problem is presented in Eq. (\ref{eq:OneStepOptimal}) in which the problem is reduced to one step lookahead with optimal cost to go as the terminal cost. The optimal cost to go is a quadratic function of $K$, the positive definite solution of the Riccati equation (\ref{eq:Riccati_infinite}), which is a function of communication network $\mc{L}$, system dynamics $B$, and communication and mobility constants $\kappa_C$ and $\kappa_M$ respectively. From Assumption (1), $\mc{L}$ is fixed, so $K$ can be computed offline before the nodes are deployed. From Prop. \ref{prop:K_Matrix}, all the entries of $K$ are non-zero. Therefore, the cost in Eq. (\ref{eq:OneStepOptimal}) has a local component and global component. Each node can compute its stage cost $g({\bf z}_k,{\bf u})$ using the current information of its neighbors only. However, the terminal cost requires information of all the nodes in the network. 

Since both the stage and the terminal costs are convex, the distributed optimization algorithm presented in \cite{Nedic} can guarantee global optimal solution by allowing communication among neighboring nodes only. The algorithm is a consensus based distributed optimization algorithm to minimize the following cost
\begin{align*}
\min \sum_{i = 1}^N f_i(x)~~~~~\text{s.t.}~~~~ x \in \field{R}^m,
\end{align*}
where $x$ is the decision vector, $f_i:\field{R}^m \rightarrow \field{R}$ is the local cost of agent $i$, and each agent can only communicate with its neighbors. To solve this problem, each agent maintains an estimate of the entire decision vector. It communicates its estimate with its neighbors, receives their estimates, and updates its estimate of the decision vector by combining the information it received form its neighbors. For computing action at time $k$, the above process is repeated for a specified number of iterations which impacts the quality of the solution. If the size of the decision vector is small, the communication overhead of the algorithm can be tolerated. However, in multiagent systems, the size of the decision vector typically depends directly on the number of nodes $N$ and can be very large. Consequently, the size of the data packet can get large which will result in significant energy overhead. 
%Reference \cite{Raffard} presents an algorithm based on dual decomposition method, but it has similar problems.      

In the proposed framework, we approximate the global component of the cost (terminal cost) with a function that is close to the original function in some sense and is locally computable. This implies that the size of the data packet for each node will depend on the cardinality of its neighborhood set. Typically, multiagent systems are sparsely connected, i.e., $|\mc{N}_i| \ll N$, so the size of the data packet is significantly reduced. \emph{The main idea proposed in this work for designing suboptimal policies for each node $i \in \{1,2,\ldots,N\}$ is to impose the communication network constraints of node $i$ on the global optimal cost matrix $K$.} The resulting matrix $K_{\mc{L}}$ will observe the constraints of the communication network, i.e., $(K_{\mc{L}})_{ij} = 0$ if $j \notin \mc{N}_i$,  and will be used to design approximate cost to go for the LQR problem. The suboptimal policy for each node is to  solve one step lookahead problem with this approximate cost to go as the terminal cost using some efficient distributed optimization algorithm. This procedure for designing suboptimal policies will guarantee that the terminal cost can be computed locally and will simplify performance comparison with the global optimal.  
 \section{Energy Aware Architectures}\label{sec:Energy-aware architectures}
In this section we propose three Energy Aware Policies (EAP)s. The first two policies are distributed in nature and the last one is decentralized. 
\begin{definition} 
A policy $\bs{\mu}$ is distributed if the control action of each agent at time $k$ is a function of its current state and the current state and input of its neighbors, i.e., $\mu_{i,k}$ depends on $z_i(k)$, ${\bf z}_{-i,k}$ and ${\bf u}_{-i,k}$.   
\end{definition}

\begin{definition}
A policy $\bs{\mu}$ is decentralized if the control action of each agent is a function of its current state and the current states of its neighbors, i.e., $\mu_{i,k}$ depends on $z_i(k)$, ${\bf z}_{-i,k}$.
\end{definition}

In a distributed policy, node $i$ will have to repeatedly communicate with its neighbors to know their current inputs. However, the current state of the neighbors in a decentralized policy can either be sensed if the nodes are equipped with the required sensors, or it can be communicated by one time communication.   

\subsection{Distributed Energy Aware Policies}
Next we propose two distributed energy aware policies EAPs I \& II. In EAP I, we formulate a semidefinite program to find a matrix that is closest to $K$ in terms of Frobenius norm and satisfies the desired sparsity structure. \newpage
\emph{Energy Aware Policy I} 

Compute the positive definite matrix $K$ offline that satisfies Riccati equation (\ref{eq:Riccati_infinite}). Solve the following semidefinite program to find the projection of $K$ on the sparsity structure of $\mc{L'} = (\mc{L}\otimes I_2)$. 
\begin{align}
&\min_{K^{\text{I}}_{\mc{L}}} \Vert K^{\text{I}}_{\mc{L}} - K \Vert_F \nonumber \\ 
& \text{s.t.      }  ~~~ K^{\text{I}}_{\mc{L}}  \geq 0, \nonumber \\
&(K^{\text{I}}_{\mc{L}})_{ij} = 0 ~~ \text{if}~~~ \mc{L'}_{ij} = 0  \nonumber \\ 
&\forall i \text{ and }j \in \{1,\ldots,2(N+2)\}.  \nonumber
\end{align}
%where $\mc{L'} = (\mc{L}\otimes I_2)$ 
%
%\vspace{0.1in}
The proposed approximate cost to go function is $$\tilde{J}^{\text{I}}({\bf z}_k)  = {\bf z}_k\tr K^{\text{I}}_{\mc{L}} {\bf z}_k,$$ and each node solves the following one stage optimization problem to compute its stationary suboptimal policy:
\begin{equation}\label{eq:policy_iteraton1}
\bs{\mu}^{\text{I}}({\bf z}_k) = \argmin_{{\bf u}}  g({\bf z}_k, {\bf u})+  \alpha \tilde{J}^{\text{I}}({\bf z}_k+B{\bf u}),
\end{equation}
with dynamics and constraints specified in \ref{eq:Problem_infinite}.

\begin{prop} 
\emph{Let $P_K$ be defined as}
\[
(P_K)_{ij} = \left\{
 \begin{array}{ll}
  K_{ij} &: \mc{L'}_{ij} \neq 0 \\
  0  &: \text{otherwise}
 \end{array}
 \right.
\]
\emph{Then} $K^{\text{I}}_{\mc{L}} = P_K$, \emph{i.e., the projection of $K$ on $\mc{L}'$ can be computed by simply replacing the entries of $K$ corresponding to non-neighboring agents with zeros. }
\end{prop} \vspace{0.1in}
\begin{proof}
Firstly, $P_K$ satisfies the desired sparsity constraints by construction. Secondly, $P_K$ is symmetric because $K$ and $\mc{L}'$ are both symmetric. To show that $P_K$ is positive definite, we use the fact that $K$ is an $M$ matrix, i.e., all the diagonal entries are positive and the off-diagonal entries are non-positive, and $K_{ii} > - \sum_{j=1}^{N+2} K_{ij}$ since $K>0$. Therefore, all the Gershgorin circles of $P_K$ lie in the positive half plane, and so all of its eigenvalues are positive. Thus, $P_k$ satisfies all the constraints. To prove that $P_K$ is optimal, we use the definition of Frobenius norm. 
\begin{equation*}
\Vert P_k - K \Vert_F = \sqrt{\sum\limits_{i=1}^{2(N+2)} \sum\limits_{j=1}^{2(N+2)} |(P_k)_{ij}-K_{ij}|^2}
\end{equation*}
For $(i,j)$ pair such that $\mc{L}'_{ij} = 0$, the value of $K^{\text{I}}_{\mc{L}}$ is fixed by the constraint. For $(i,j)$ such that $\mc{L}'_{ij} \neq 0$, any value for $(K^{\text{I}}_{\mc{L}})_{ij}$ other than $(P_K)_{ij} = K_{ij}$ will result in a positive contribution in the error term, so $(K^{\text{I}}_{\mc{L}})_{ij} = K_{ij}$ is optimal, which concludes the proof.  
\end{proof}
Therefore, we can decompose the matrix $K$ into a sum of two matrices
\begin{equation*}
K = K^{\text{I}}_{\mc{L}} + K^{\text{I}}_{\mc{L}^c}, 
\end{equation*}
where 
\[(K^{\text{I}}_{\mc{L}^c})_{ij} = \left\{
  \begin{array}{ll}
    K_{ij} &:  \mc{L'}_{ij} = 0 \\
    0 &:  \text{otherwise}
 \end{array}
\right.
\]
% 
%The proposed approximate cost to go function at $k$ is $\tilde{J}^{\text{I}}({\bf z}_k)  = {\bf z}_k\tr K^{\text{I}}_{\mc{L}} {\bf z}_k$ and each node solves the following one stage optimization problem to compute its stationary suboptimal policy:
%
%
%\begin{equation}\label{eq:policy_iteraton1}
%\bs{\mu}^{\text{I}}({\bf z}_k) = \argmin_{{\bf u}}  g({\bf z}_k, {\bf u})+  \alpha \tilde{J}^{\text{I}}({\bf z}_k+B{\bf %u}),
%\end{equation}
%
%with dynamics and constraints specified in \ref{eq:Problem_infinite}.\vspace{0.1in}
%
%
\begin{prop} \label{prop:EAPI_dist}
\emph{EAP I with approximate cost to go }
\begin{equation}\label{eq:Policy1}
\tilde{J}^{\text{I}}({\bf z}_k) = {\bf z}_k\tr K^{\text{I}}_{\mc{L}}{\bf z}_k.
\end{equation}
\emph{is a distributed policy, i.e., $\mu^{\text{I}}_{i}$ depends on $z_i(k)$, ${\bf z}_{-i,k}$, and ${{\bf u}}_{-i,k}$}
\end{prop}
\vspace{0.05in}
\begin{proof}
The stage cost $g({\bf z}_k, \bs{\mu}^{\text{I}}({\bf z}_k))$ can be expanded into
\begin{align}\label{eq:stageCost_agent}
g({\bf z}_k, \bs{\mu}^{\text{I}}({\bf z}_k)) &= \sum_{i = 1}^N g_i(z_i(k), {\bf z}_{-i,k},{\mu}^{\text{I}}_{i}(z_i(k))) \nonumber \\
%g_i(z_i(k), {\bf z}_{-i,k},{\mu}^{\text{I}}_{i}(z_i(k)))&= \frac{1}{2}\sum_{j \in \mc{N}_i} \kappa_C \Vert z_i(k) - z_j(k)\Vert ^2 + \ldots \nonumber \\
 %&~~~~~~~~~+\kappa_M \Vert \mu^{\text{I}}_{i}(z_i(k)) \Vert^2.
 g_i(z_i(k), {\bf z}_{-i,k},{\mu}^{\text{I}}_{i}(z_i(k)))&= \frac{1}{2}\sum_{j \in \mc{N}_i} \kappa_C \Vert z_i(k) - z_j(k)\Vert ^2 + \kappa_M \Vert \mu^{\text{I}}_{i}(z_i(k)) \Vert^2.
\end{align}
From the definition of $K^{\text{I}}_{\mc{L}}$ and using the fact that matrix $K$ is a positive definite $M$ matrix, (\ref{eq:Policy1}) can be written as 
\begin{align}\label{eq:costToGo_agent}
\tilde{J}^{\text{I}}({\bf z}_k) &= \sum_{i = 1}^N \tilde{J}^{\text{I}}_i(z_i(k),{\bf z}_{-i,k}), \nonumber \\
\tilde{J}^{\text{I}}_i(z_i(k),{\bf z}_{-i,k}) &=  \sum_{j \in \mc{N}_i} \vert(K^{\text{I}}_{\mc{L}} )_{ij}\vert \Vert z_i(k) - z_j(k)\Vert^2 +
\left((K^{\text{I}}_{\mc{L}} )_{ii}-\sum_{j \in \mc{N}_i}|(K^{\text{I}}_{\mc{L}} )_{ij}|\right) \Vert z_i(k)\Vert^2.
\end{align}
In the above equation, the second summation appears because the terms corresponding to non-neighboring nodes of $i$ are set equal to zero in the definition of $K^{\text{I}}_{\mc{L}}$. From Eqs. (\ref{eq:stageCost_agent}) and (\ref{eq:costToGo_agent}), it is straightforward that when solving (\ref{eq:policy_iteraton1}), each node will effectively be solving 
\begin{align}\label{eq:policy1_agent}
 {\mu}^{\text{I}}_{i}(z_i(k),{\bf z}_{-i,k}) = \argmin_{u_i}  g_i(z_i(k),{\bf z}_{-i,k},u_i)+  \alpha \tilde{J}^{\text{I}}_i(z_i(k+1),{\bf z}_{-i,k+1}),
\end{align}
where $z_i(k+1) = z_i(k) + B u_i$. To solve the above problem, node $i$ only requires its local information, i.e., its own states and the states and inputs of its neighbors.
\end{proof}

To solve (\ref{eq:policy1_agent}), node $i$ needs to know ${\bf u}_{-i,k}$ as $\tilde{J}^{\text{I}}_i(z_i(k+1),{\bf z}_{-i,k+1})$ depends on it. This can be accomplished by implementing an efficient distributed optimization algorithm that can ensure that each node finally has an accurate estimate of the control inputs, ${\bf u}_{-i,k}$, that its neighbors will implement. The algorithm that we used  for simulation in the next section is a distributed subgradient algorithm that was presented in \cite{Nedic}. In this algorithm, each node maintains an estimate ${\bf \hat{u}}_i $ of the optimization variables that it needs to compute its cost. For our problem, at time $k$, ${\bf \hat{u}}_i \in \field{R}^{2(|\mc{N}_i|+1)}$ is node $i$'s estimate of $u_i(k)$ and ${\bf u}_{-i,k}$. In particular, if $\mc{N}_i = \{i_1,\ldots,i_{|\mc{N}_i|}\}$, then ${\bf \hat{u}}_i = [\hat{u}_i^i ~~{\bf \hat{u}}_{-i}] $ where  ${\hat{u}}_i^j$ is node $i$'s estimate of $u_j(k)$ for $j \in \{\mc{N}_i \cup i\}$ and the vector ${\bf \hat{u}}_{-i} = [\hat{u}_i^{i_1 }~~\hat{u}_i^{i_2} \ldots \hat{u}_i^{|\mc{N}_i|}]$. 
The main idea is the use of consensus to ensure that the estimates of the neighboring nodes converge to same values. To compute $u_i(k)$ and ${\bf u}_{-i,k}$, node $i$ performs the following steps:
\newpage
\emph{Distributed Projected Subgradient Algorithm \cite{Nedic}}
\begin{description}
\item[0] At iteration 0, node $i$ initializes its estimate vector ${\bf \hat{u}}_i $ with some random values.  
\item [1] At iteration $m$, node $i$ updates its estimate for all $j \in \{\mc{N}_i \cup i\}$ as follows:
%\begin{align}
% v_i(m) &= \sum_{j = 1}^N a_i^j(m) u^j(m), \nonumber \\ 
%\label{eq:estUpdate}\hat{ u}_i(m+1) &= P_{U_i}[{v}_i(m) - \gamma_k d_i(m)]
%\end{align}
\[v_i^j = \left\{
  \begin{array}{ll}
  \frac{1}{2}(\hat{u}_i^j(m) + \hat{u}_j^j(m) ) &:  i \neq j \\
   \sum_{j \in \{\mc{N}_i \cup i\}}  \frac{1}{|\mc{N}_i | +1} \hat{u}_j^i(m)&:  i = j
   \end{array}
\right.
\]
\begin{align}\label{eq:estUpdate}
\hat{ u}^j_i(m+1) &= v_i^j - \gamma d_i(m).
\end{align}
\item[2] Repeat while $m \leq \text{iter}$.
\item[3] $u_i(k) = \hat{u}_i^i(\text{iter})$. 
\end{description}
In this algorithm, ``iter" is the total number of iterations of the algorithm, and $\gamma$ is the step size of the descent.
In Eq. (\ref{eq:estUpdate}), $d_i(m)$ is the gradient of the cost function of node $i$ evaluated at ${\bf \hat{u}}_i(m)$. To update its estimate, $\hat{ u}^j_i(m+1)$, node $i$ first combines the estimates from its neighbors, which is the consensus step in Eq. (\ref{eq:estUpdate}). Then it computes the gradient of its cost at ${\bf \hat{u}}_i(m)$ and updates it estimate by moving the consensus value towards the negative of the gradient. 

To compute its cost to go, each node exchanges its estimates of control values with its neighbors ``iter" times for all $k$. Although the proposed scheme has communication overhead, it is small since each node is only communicating its estimates of the current control values of itself and its neighbors. In some of the existing schemes for similar problems, nodes communicate their entire control and state trajectories with their neighbors which result in huge communication overhead depending on the horizon length and number of nodes in the network (\cite{Dunbar} and \cite{Keviczky}). 

In EAP I, we simply imposed the sparsity structure of the communication network on the optimal cost to go matrix $K$. The resulting approximate terminal cost in Eq.  (\ref{eq:costToGo_agent}) had one summation that consisted of the square of the distances of nodes $i$ and $j$, such that $(i,j) \in E$, weighted with $K_{ij}$. Those were the desired terms because the objective in Problem $\mc{P}$ is to minimize the distances between the neighboring nodes. However, Eq. (\ref{eq:costToGo_agent}) had a second summation which was 
\begin{equation*}
\left((K^{\text{I}}_{\mc{L}} )_{ii}-\sum_{j \in \mc{N}_i}|(K^{\text{I}}_{\mc{L}} )_{ij}| \right)\Vert z_i(k)\Vert^2.
\end{equation*}
One way to interpret the above term is that each node is trying to minimize its distance from the origin. Since base station $\mc{B}_1$ is assumed to be located at the origin, minimizing distance from the origin can be modeled by adding an edge between each node $i$ and $\mc{B}_1$. Therefore, the approximate terminal cost in EAP I is with respect to a new graph $G^{I}(V,E^I)$ where $E^{I} = E \cup S $ where $S = \{(i,N+1) \text{ for all  } i\in \{1,2,\ldots,N\}\}$ and $N+1$ is the index of $\mc{B}_1$. This remodeling of the structure of the system can have serious consequences that are evident in system simulation in Section \ref{sec:simulation} Fig. \ref{fig:trajectories}. Figures \ref{subfig:Traj_optimal} and \ref{subfig:Traj_EAPI} are the trajectories of relay nodes under the optimal policy and EAP I respectively. By comparing these trajectories, it is obvious that under EAP I, the trajectories of the nodes are biased towards the origin, which results in significant increase in the total cost of the system.   

Next we propose EAP II, which introduces a refined projection of the optimal cost that removes the undesired terms from the resulting approximate cost to go. The refined projection will improve performance as will be shown in Section \ref{sec:simulation} via simulations. \vspace{0.1in}\\
\emph{Energy Aware Policy II}\\
Compute the positive definite matrix $K$ offline that satisfies Riccati equation (\ref{eq:Riccati_infinite}). Decompose the matrix $K$ into sum of two matrices.
\begin{equation*}
K = K^{\text{II}}_{\mc{L}} + K^{\text{II}}_{\mc{L}^c}, 
\end{equation*}
such that 
\[(K^{\text{II}}_{\mc{L}})_{ij} = \left\{
  \begin{array}{ll}
    K_{ij} &: i\neq j \text{ and } \mc{L}'_{ij} \neq 0 \\
    K_{ii} - \sum_{p \not \in \mc{N}_i} K_{ip} &: i = j   \\
    0  &: \text{otherwise}
  \end{array}
\right.
\]
and 
\[(K^{\text{II}}_{\mc{L}^c})_{ij} = \left\{
  \begin{array}{ll}
    K_{ij} &:  i\neq j \text{ and } \mc{L}'_{ij} = 0 \\
    \sum_{p \not \in \mc{N}_i}K_{ip} &:    i = j\\
    0 &:  \text{otherwise}
  \end{array}
\right.
\]
The proposed approximate cost to go function at time $k$ is $$\tilde{J}^{\text{II}}({\bf z}_k)  = {\bf z}_k\tr K^{\text{II}}_{\mc{L}} {\bf z}_k,$$ where $K^{\text{II}}_{\mc{L}}$ is a refined projection of $K$ on the sparsity structure of $\mc{L}'$. Each agent solves the following one stage optimization problem to compute its stationary suboptimal policy:
\begin{equation}\label{eq:policy_iteration2}
\bs{\mu}^{\text{II}}({\bf z}_k)= \argmin_{{\bf u}}  g({\bf z}_k, {\bf u})+  \alpha \tilde{J}^{\text{II}}({\bf z}_k+B {\bf u}),
\end{equation}
with dynamics and constraints specified in \ref{eq:Problem_infinite}.\vspace{0.05in}
\begin{prop}\label{prop:EAPII_dist}
\emph{EAP II with approximate cost to go function at time $k$}
\begin{equation}\label{eq:Policy2}
\tilde{J}^{\text{II}}({\bf z}_k) = {\bf z}_k\tr K^{\text{II}}_{\mc{L}}{\bf z}_k
\end{equation}
\emph{is a distributed policy, i.e., $\mu^{\text{II}}_{i,k}$ depends on $z_i(k)$, $\textnormal{{\bf z}}_{-i,k}$, and $\textnormal{{\bf u}}_{-i,k}$. }
\end{prop}
 \vspace{0.05in} 
\begin{proof}
Using the same argument as in Prop. \ref{prop:EAPI_dist}, stage cost $g({\bf z}_k, \bs{\mu}^{\text{II}}({\bf z}_k))$ is decentralized. For cost to go, using the definition of $K^{\text{II}}_{\mc{L}}$ and the fact that matrix $K$ is a positive definite $M$ matrix, Eq. (\ref{eq:Policy2}) can be written as 
\begin{align*}
\tilde{J}^{\text{II}}({\bf z}_k) &= \sum_{i = 1}^N \tilde{J}^{\text{II}}_i(z_i(k),{\bf z}_{-i,k}), \nonumber \\
\tilde{J}^{\text{II}}_i(z_i(k),{\bf z}_{-i,k})&= \sum_{j \in \mc{N}_i} |(K^{\text{II}}_{\mc{L}})_{ij}| \Vert z_i(k) - z_j(k)\Vert^2 .
\end{align*}
In the above equation, the second summation consisting of the undesirable terms in Eq. (\ref{eq:costToGo_agent}) does not appear anymore because of the correction introduced in the definition of the refined projection matrix. It is obvious from the above arguments that when solving (\ref{eq:policy_iteration2}), each agent will effectively be solving 
\begin{align*}%\label{eq:policy2_agent}
{\mu}^{\text{II}}_i(z_i(k),{\bf z}_{-i,k}) = \argmin_{u_i}  g_i(z_i(k),{\bf z}_{-i,k},u_i)+  \alpha J^{\text{II}}(z_i(k+1),{\bf z}_{-i,k+1}).
\end{align*}
To solve the above problem, node $i$ only requires its own information and the information of its neighbors.
\end{proof}
Each node can compute its control action $\mu_i^{\text{II}}(z_i(k))$ by implementing the same distributed optimization algorithm presented for EAP I. \vspace{0.05in}

\subsection{Performance Analysis}
Next we analyze the performance of the policies presented in the previous section. However, the analysis carried out in this section is for a more general system in which there are $m$ base stations, and $N$ mobile nodes have to establish communication links between these base stations. Furthermore, we analyze system performance for an entire class of suboptimal policies. To summarize, we are interested in the analysis of the following one step look-ahead optimization problem:
\begin{align}\label{eq:approxProb}
%\hat{J}({\bf z}_k) &= 
&\min\limits_{{\bf u}} g({\bf z},{\bf u}) + \alpha \tilde{J}({\bf z}^+)\nonumber\\
&\text{s.t.} ~~~~~~{\bf z}^{+} = {\bf z} + B{\bf u},
\end{align}
where $g({\bf z},{\bf u})$ is defined in Eq. (\ref{eq:cost_stage_terminal}). Here, $$\tilde{J}({\bf z}) = {\bf z}\tr H {\bf z}$$ is an approximate cost to go, $H$ is any symmetric positive semi-definite matrix that satisfies the constraints of communication network, and $B = \left[  \begin{array}{c} 
I_{2N} \\
\hline
0_{2m\times 2N}  
\end{array} \right].
$ 
Since the problem is a one stage LQR problem, the optimal policy is $\hat{\bs{\mu}}({\bf z}) = -\hat{L} {\bf z}$,
\begin{equation}\label{eq:L_hat}
\hat{L} = \alpha(\alpha B\tr H B + R)^{-1}B\tr H,
\end{equation}
and the optimal cost for the approximate problem with one step look ahead is $\hat{J}({\bf z}) ={\bf z}\tr \hat{K} {\bf z}$
\begin{equation*}\label{eq:K_hat}
\hat{K} = \alpha H - \alpha^2 HB(\alpha B\tr HB+R)^{-1}B\tr H + Q.
\end{equation*}
We start the analysis by proving that the optimal policy $\hat{\bs{\mu}}$ results in stabilizing system dynamics. For analysis purposes, we use the following matrix partitioning
\begin{equation}\label{eq:partitionMatrix}
U=
\left[
\begin{array}{c|c}
U_f & U_{fl} \\
\hline
U_{lf} & U_{ll}
\end{array}
\right],
\end{equation}
where $U \in \field{R}^{2(N+m)\times2(N+m)}$, $U_f \in \field{R}^{2N \times 2N}$, $U_{fl} \in \field{R}^{2N \times 2m}$, $U_{lf} \in \field{R}^{2m \times 2N}$, and $U_{ll} \in \field{R}^{2m \times 2m}$. If $U$ is symmetric then $U_{lf} = U_{fl}\tr$. 

\vspace{0.05in}
\begin{lem} \label{prop:stability}
\emph{Let $H$ be a symmetric and positive semidefinite matrix, and let} $\tilde{J}({\bf z}) ={\bf z}\tr H {\bf z}.$ \emph{Then the system dynamics}
\begin{equation*}
{\bf z}^+ = (I_{2(N+m)} - B \hat{L}){\bf z}
\end{equation*}
\emph{are marginally stable, where $\hat{L}$ is defined in Eq. (\ref{eq:L_hat}).}
% $\tilde{H} = \alpha B(\alpha  B\tr H B + R)^{-1}B\tr H A$, ${\bf x} \in \field{R}^N$, $A = I_N$, $R=c I_m$$(c > 0)$ \emph{and }
\end{lem}
\begin{proof}
Let $\tilde{H} = B \hat{L}$ and $\lambda_i = \lambda_i(I_{2(N+m)}  - \tilde{H})$ be the $i^{th}$ eigenvalue of $(I_{2(N+m)} - \tilde{H})$ such that $\lambda_1 \geq \lambda_2\geq \ldots\geq \lambda_{2(N+m)}$. To prove that the system is marginally stable, we need to show that $| \lambda_i | \leq 1$ for all $i = \{1,2,\ldots,2(N+m)\}$ and $(I_{2(N+m)} -\tilde{H})$ has $2(N+m)$ independent eigenvectors. By partitioning $H$ as in Eq. (\ref{eq:partitionMatrix}), $\alpha B\tr H B= \alpha H_f$. Thus,
\[
\tilde{H}=
\left[
\begin{array}{c|c}
\alpha (\alpha H_f + R)^{-1} H_f& \alpha ( \alpha H_f + R)^{-1} H_{fl}\\
\hline
0_{2m\times 2N} &  0_{2m\times 2m}
\end{array}
\right].
\]
Using the properties of block matrices, the eigenvalues of $\tilde{H}$ are the eigenvalues of $\alpha( \alpha H_f + R)^{-1}H_f$ and $0_{2m\times 2m}$. Therefore, $0_{2m\times 2m}$ contributes $2m$ zero eigenvalues and $\alpha(\alpha H_f + R)^{-1} H_f$ contributes $2N$ eigenvalues. Next we will show that these $2N$ eigenvalues are real, positive and less then one. 

Since $H$ is symmetric and positive semidefinite, $H_f = B\tr H B$ is also symmetric and positive semidefinite. Therefore, $H_f$ has real and non-negative eigenvalues and $2N$ independent eigenvectors. To show that the eigenvalues are also less then one, let $\hat{\lambda}_i$ be the $i^{th}$ eigenvalue of $H_f$. Then $\frac{1}{ \alpha \hat{\lambda}_i+\kappa_M}$ and $\frac{\alpha \hat{\lambda}_i}{\alpha \hat{\lambda}_i+\kappa_M}$are the corresponding eigenvalues of $(\alpha H_f + R)^{-1}$ and  $ (\alpha H_f + R)^{-1}\alpha H_f$ respectively. Here we have used the fact that if two matrices $P$ and $Q$ have the same set of eigenvectors, then they commute and $\lambda_i(PQ) = \lambda_i(P)\lambda_i(Q)$. Therefore, the eigenvalues of $(\alpha H_f + R)^{-1}\alpha H_f$ are always less than or equal to one for $\kappa_M>0$. This ensures that the eigenvalues of $I_{2N} - \alpha( H_f + R)^{-1}\alpha H_f$ are also less than or equal to one. Since the last $2m$ rows of $\tilde{H}$ are zero, $I_{2m} - \tilde{H}_{ll} =I_{2m}$ and there will be $2m$ more independent eigenvectors. This implies that the eigenvalues of $I_{2(N+m)} - \tilde{H}$ are always less then or equal to one with independent eigenvectors, which concludes the proof. 
\end{proof}
\vspace{0.05in}
We can now analyze the stability properties of the proposed policies based on Lem. \ref{prop:stability}. Since $K^{\text{I}}_{\mc{L}}$ is positive definite and $K^{\text{II}}_{\mc{L}}$ is positive semidefinite, the dynamics for EAP I \& II are  stable. 

An important consequence of Lem. \ref{prop:stability} is that given the initial locations of the mobile relay nodes, the state space of the system is bounded since the dynamics are stable. Therefore, the performance analysis of the system can be restricted to a bounded set $\mc{S}$ that is invariant under the system dynamics. We say that a set $\mc{S}$ is invariant under policy $\bs{\mu}$ if ${\bf z} \in \mc{S}$ implies that ${\bf z}^+ = {\bf z} + B \bs{{\mu}}({\bf z}) \in \mc{S}$. We define the max norm of a function $J$ over a set $\mc{S}$ by 
\begin{align*}
\Vert J \Vert_{\mc{S}} = \max_{{\bf z} \in \mc{S}} \vert J(z) \vert.
\end{align*}
For error analysis we use two mappings from \cite{Bertsekas}. 
%For any function $J: \mc{S} \rightarrow \field{R}$ where $\mc{S}$ is some closed and bounded set invariant under $\bs{\mu}$ and $\bs{\mu}^*$, the mappings $T$ and $T_{\mu}$ are 
Let $\mc{S}_{\hat{\mu}}$ be an invariant set under policy $\bs{\hat{\mu}}$, i.e., if ${\bf z} \in \mc{S}_{\hat{\mu}}$ then ${\bf z}^+ \in \mc{S}_{\hat{\mu}}$ where ${\bf z}^+ = {\bf z} + B \bs{\hat{\mu}}({\bf z})$. Let $\mc{S}^*$ be the minimum set such that $\mc{S}_{\hat{\mu}} \subseteq \mc{S}^*$ and ${\bf z} \in \mc{S}_{\hat{\mu}}$ implies that ${\bf z}^+ = {\bf z} + B \bs{\mu}^*({\bf z}) \in \mc{S}^*$. For any function $J: \mc{S}^* \rightarrow \field{R}$, the mappings $T$ and $T_{\hat{\mu}}$ are such that $(TJ):\mc{S}_{\hat{\mu}} \rightarrow \field{R}$ and $T_{\hat{\mu}} : \mc{S}_{\hat{\mu}} \rightarrow \field{R}$ and are defined as
\begin{align}\label{eq:Map_T_Tmu}
(TJ)({\bf z}) &= \min\limits_{{\bf u}} g({\bf z},{\bf u}) + \alpha {J}(A{\bf z}+ B {\bf u}), \nonumber \\
(T_{\hat{\mu}}J)({\bf z}) &= g({\bf z},\bs{\mu}({\bf z})) + \alpha {J}(A{\bf z}+ B \bs{\mu}({\bf z})).
\end{align}
Let $T^k$ and $T_{\hat{\mu}}^k$ be the composition of the mappings $T$ and $T_{\hat{\mu}}$ with themselves $k$ times respectively, i.e., 
\begin{align*}
T^kJ &= T(T^{k-1}J), \\
T_{\hat{\mu}}^k J&= T_{\hat{\mu}}(T_{\mu}^{k-1}J).
\end{align*}
For a stationary policy $\hat{\mu}$, the associated cost to go function $J^{\hat{\mu}}$ satisfies  $\lim_{k\to \infty} T_{\hat{\mu}}^k J $. 
It has been proved in \cite{Bertsekas} that the optimal cost to go $J^*$ satisfies Bellman equation $J^* = TJ^*$. Similarly, for a stable policy $\bs{\hat{\mu}}$, $J^{\hat{\mu}} = T_{\hat{\mu}}J^{\hat{\mu}}.$ A mapping $T$ is a contraction mapping if there exists a scalar $\beta <1$ such that
\begin{align*}
\Vert TJ - T\bar{J} \Vert_{\mc{S}} \leq \beta\Vert J - \bar{J} \Vert_{\mc{S}}
\end{align*}
The monotonicity lemma (Lem. 2.1 in \cite{Bertsekas}) states that for any two functions $J$ and $\bar{J}$ defined on $\mc{S}^*$ such that 
$$J({\bf z}) \leq \bar{J}({\bf z}) ~~~~~~\text{for all}~~~{\bf z}\in \mc{S}^*,$$
the following inequalities hold:
\begin{align*}
%J({\bf z}) &\leq \bar{J}({\bf z}) ~~~~~~\text{for all}~~~{\bf z}\in \mc{S}^*,\\
(TJ)({\bf z}) &\leq (T \bar{J})({\bf z})~~~\text{for all}~~ {\bf z}\in \mc{S}_{\hat{\mu}}, \\
(T_{\hat{\mu}}J)({\bf z}) &\leq (T_{\hat{\mu}} \bar{J})({\bf z})~~~\text{for all}~~ {\bf z}\in \mc{S}_{\hat{\mu}}.
\end{align*}
Finally, for the sets $\mc{S}_{\hat{\mu}}$ and $\mc{S}^*$ as defined above,
%Let $\mc{S}_{\hat{\mu}}$ be an invariant set under policy $\bs{\hat{\mu}}$, i.e., if ${\bf z} \in \mc{S}_{\hat{\mu}}$ then ${\bf z}^+ \in \mc{S}_{\hat{\mu}}$ where ${\bf z}^+ = {\bf z} + B \bs{\hat{\mu}}({\bf z})$. Let $\mc{S}^*$ be the minimum set such that $\mc{S}_{\hat{\mu}} \subseteq \mc{S}^*$ and ${\bf z} \in \mc{S}_{\hat{\mu}}$ implies that ${\bf z}^+ = {\bf z} + B \bs{\mu}^*({\bf z}) \in \mc{S}^*$.

%\begin{align*}
%\mc{S}^{*} &=\mc{S}_{\hat{\mu}} \cup \mc{S} \\
%\mc{S} &= \{{\bf z} \in \field{R}^{2(N+m)} |~{\bf y} \in \mc{S}_{\hat{\mu}} \implies {\bf z} =({\bf y} + B\bs{\mu}^*({\bf y}) ) \in \mc{S}^{*} \},
%\end{align*}
$$\Vert T_{\hat{\mu}} \tilde{J} - T J^*\Vert_{\mc{S}_{\hat{\mu}}} \leq \alpha \Vert \tilde{J} - J^* \Vert_{\mc{S}^*}.$$ To show this, the first step is to recognize that $\bs{\hat{\mu}}$ is the greedy policy with terminal cost $\tilde J$, so $(T_{\hat{\mu}} \tilde{J}) = (T \tilde{J}).$ Let $c = \Vert \tilde{J} - J^* \Vert_{\mc{S}^{*}}$. Then 
\begin{align*}
&\tilde{J}({\bf z}) - c \leq J^*({\bf z}) \leq \tilde{J}({\bf z}) + c ~~~\forall ~~{\bf z} \in {\mc{S}^*} \\
&\tilde{J}({\bf z}) - c \leq J^*({\bf z}) \leq \tilde{J}({\bf z}) + c ~~~\forall ~~{\bf z} \in \mc{S}_{\hat{\mu}}
\end{align*}
The second set of inequalities hold because $\mc{S}_{\hat{\mu}} \subseteq \mc{S}^{*}$. For  
\begin{align*}
(T(\tilde{J} - c))({\bf z}) &= \min_{{\bf u}} [g({\bf z},{\bf u})+\alpha (\tilde{J}({\bf z}^+) - c)]\\
%(T(\tilde{J} - c))({\bf z}) 
&= (T \tilde{J})({\bf z}) - \alpha c.
\end{align*}
Similarly $(T(\tilde{J} +c))({\bf z}) = (T \tilde{J})({\bf z}) + \alpha c$. Using the monotonicity property of $T$, 
\begin{align*}
T\tilde{J}({\bf z}) - \alpha c \leq (TJ)^*({\bf z}) \leq T\tilde{J}({\bf z}) + \alpha c ~~~\forall ~~{\bf z} \in \mc{S}_{\hat{\mu}},
\end{align*}
which implies that 
\begin{align}\label{eq:Tmu-T}
\Vert T_{\hat{\mu}} \tilde{J} - T J^*\Vert_{\mc{S}_{\hat{\mu}}} \leq \alpha c = \alpha \Vert \tilde J - J^* \Vert_{\mc{S}^*}
\end{align}

Next we analyze the performance of any approximate policy by comparing it with the global optimal policy. We will derive bound for maximum error between the optimal and a suboptimal policy. 

\begin{thm} \label{prop:ErrorI}
{\emph{Let $H$ be a symmetric and positive semidefinite matrix such that} $\tilde{J}({\bf z}) ={\bf z}\tr H {\bf z}$ \emph{is the approximate cost to go in (\ref{eq:approxProb}). Let $$\epsilon = \Vert \tilde{J} - J^* \Vert_{\mc{S}^*},$$ and $J^{\hat{\mu}} = \lim_{k\to \infty} T_{\hat{\mu}}^k \tilde{J}$. Then the maximum error between the global optimal solution and the approximate solution is}
\begin{equation}\label{eq:maxErrorBound}
\Vert J^{\hat{\mu}} - J^* \Vert_{\mc{S}_{\hat{\mu}}} \leq \frac{2\alpha \epsilon}{1-\alpha}
\end{equation}

}
\end{thm}
\vspace{0.05in}
\begin{proof}
The proof of this theorem is based on the properties of the mappings defined in Eq. (\ref{eq:Map_T_Tmu}). 
\begin{align*}
\Vert J^{\hat{\mu}} - J^* \Vert_{\mc{S}_{\hat{\mu}}} &=   \Vert T_{\hat{\mu}}J^{\hat{\mu}} - J^* \Vert_{\mc{S}_{\hat{\mu}}}  \nonumber\\
&\leq \Vert T_{\hat{\mu}}J^{\hat{\mu}} - T_{\hat{\mu}}\tilde{J} \Vert_{\mc{S}_{\hat{\mu}}} + \Vert T_{\hat{\mu}}\tilde{J}- TJ^* \Vert_{\mc{S}_{\hat{\mu}}}\nonumber \\
&\leq \alpha \Vert J^{\hat{\mu}} - \tilde{J} \Vert_{\mc{S}_{\hat{\mu}}} + \alpha \Vert \tilde{J}- J^* \Vert_{\mc{S}^*}\nonumber \\
&\leq \alpha \Vert J^{\hat{\mu}} - J^* \Vert_{\mc{S}_{\hat{\mu}}}+\alpha \Vert J^* - \tilde{J} \Vert_{\mc{S}_{\hat{\mu}}} \nonumber \\
&+ \alpha \Vert \tilde{J}- J^* \Vert_{\mc{S}^*} \\ \nonumber
&\leq \alpha \Vert J^{\hat{\mu}} - J^* \Vert_{\mc{S}_{\hat{\mu}}}+ 2\alpha \epsilon,\nonumber
\end{align*}
which concludes the proof. Here we have used the fact the both $T$ and $T_{\hat{\mu}}$ are contractions and the result proved in Eq. (\ref{eq:Tmu-T}).
\end{proof}

One important advantage of using EAPs I $\&$ II is that they simplify the computation of these error bounds. For any general suboptimal cost $\tilde{J}$, it is not straightforward to compute $\epsilon$. However for both EAPs I $\&$ II, their corresponding values of $\epsilon$ can easily be computed as follows: 
\begin{align*}
\epsilon^{\text{I}} = \Vert \tilde{J}^{\text{I}} - J^* \Vert_{\mc{S}^*} = \max_{{\bf z}\in \mc{S}^*} |{\bf z}\tr K_{\mc{L}^c}^{\text{I}} {\bf z}|, \\
\epsilon^{\text{II}} = \Vert \tilde{J}^{\text{II}} - J^* \Vert_{\mc{S}^*} = \max_{{\bf z}\in \mc{S}^*} |{\bf z}\tr K_{\mc{L}^c}^{\text{II}} {\bf z}|
\end{align*}
In fact, we can derive tight upper bounds for both $\epsilon^{\text{I}} $ and $\epsilon^{\text{II}} $. Let 
\begin{align*}
\zeta_i &= \sum_{j\notin \mc{N}_i} K_{ij}, \text{ and }\\
\zeta_{\max} &=\max\limits_{i =\{1,2,\ldots,N+2\}} \zeta_i. 
\end{align*}
Then from Gershgorin circle theorem, 
\begin{align*}
\epsilon^{\text{I}}  &\leq \zeta_{\max} \max_{{\bf z}\in \mc{S}^*} \Vert {\bf z} \Vert ^2, \nonumber\\
\epsilon^{\text{II}}  &\leq 2\zeta_{\max} \max_{{\bf z}\in \mc{S}^*} \Vert {\bf z} \Vert ^2.
\end{align*}
Here $\zeta_i$ is the sum of the weights assigned to the links between node $i$ and its non-neighboring nodes in the global optimal cost to go.
%\vspace{0.05in}

\subsection{Decentralized Energy Aware Policy}
EAPs I $\&$ II are distributed because each node is able to compute its control action by communicating with its neighbors only. The communication overhead is small especially in the context of this problem setup in which each agent has at most two neighbors. However, in networks with dense deployment of nodes, even this communication can cause significant energy consumption, and can result in channel congestion if all the nodes transmit simultaneously. To prevent congestion, nodes need to come up with some scheduling scheme. However, any scheduling scheme will have its own cost and will introduce latency in the system. Therefore, it is desirable to have a coordination policy that requires no inter-agent communication and can still provide some performance guarantees. We call such a policy a decentralized policy. Inter-agent communication can be avoided if the nodes are equipped with sensors that can sense the required information of the neighbors. In the absence of such sensors, a decentralized coordination policy should only require a node to communicate with its neighbors once to get their current state information. Next we propose a simple decentralized scheme that satisfies these requirements and can be implemented efficiently. \\

\hspace{-0.15in}\emph{Energy Aware Policy III} \vspace{0.05in}\\
In EAP III, at time $k$, node $i$ assumes that $u_i(t) = 0$ for all $t\geq k+1$ and $u_j(t) = 0$ for all $t\geq k$ where $j \in \mc{N}_i$. Therefore, the total cost of the system is $J({\bf z}_k,{\bf u}_k) = \sum_{i = 1}^N J_i(z_i(k), {\bf z}_{-i,k},u_i(k))$, where 
\vspace{0.1in}
\begin{align*}%\label{eq:cost_policy3}
J_i(z_i(k), {\bf z}_{-i,k},u_i(k)) = \frac{1}{2}  \sum_{j \in \mc{N}_i}  \kappa_C\Vert z_i(k) - z_j(k)\Vert^2 + \kappa_M\Vert u_i(k) \Vert^2
+ \frac{1}{2}\sum_{t = 1}^{\infty}\alpha^t \sum_{j \in \mc{N}_i}\kappa_C \Vert z_i(k+1) - z_j(k)\Vert^2.
\end{align*}
\vspace{0.1in} Since the cost is convex, the optimal control $u_i(k)$ can easily be computed via first order necessary condition, i.e., $\frac{\partial J}{\partial u_i(k)} = 0$, which yields 
\begin{equation*}
\mu_i^{\text{III}} (z_i(k),{\bf z}_{-i,k})= \frac{-1}{\frac{2(1-\alpha)}{\alpha}\frac{\kappa_M}{\kappa_C} + |\mc{N}_i|} \sum_{j \in \mc{N}_i}(z_i(k) - z_j(k)).
\end{equation*}
This means that each agent will have consensus dynamics. The standard form of the problem under EAP III is 
\begin{align*}%\label{eq:policy2_agent}
\bs{\mu}^{\text{III}}({\bf z}_k) &= \argmin_{{\bf u}}  g({\bf z}_k, {\bf u})+ \alpha \tilde{J}^{\text{III}}({\bf z}_k,{\bf u}),\nonumber \\
\text{s.t.}~~~~ &{\bf z}_{k+1} = {\bf z}_k + B{\bf u} .
\end{align*}
The stage and the terminal costs are
\begin{align}\label{eq:ApproxCostIII}
g({\bf z}_k,{\bf u}) &= {\bf z}_k\tr Q {\bf z}_k + {\bf u}\tr R {\bf u}, \nonumber\\
\tilde{J}^{\text{III}}({\bf z}_k,{\bf u}) &= {\bf z}_k\tr \tilde{Q} {\bf z}_k + {\bf u}\tr \tilde{R} {\bf u} + {\bf u}\tr B\tr \tilde{L} {\bf z}_k,
\end{align}
%$\tilde{Q} = c_1 (\mc{L}\otimes I_2)$, $\tilde{R} = |\mc{N}_i| c_1 I_{2N}$, and $\tilde{L} = 2c_1 (\mc{L}\otimes I_2)$. The constant $c_1$ is 
$\tilde{Q} = c_1 \mc{L}'$, $\tilde{R} =  c_1  (R' \otimes I_2)$, and $\tilde{L} = 2c_1 \mc{L}'$. Here $\mc{L}' = \mc{L}\otimes I_2$, $R'$ is an $N \times N$ diagonal matrix with $R'_{ii} = |\mc{N}_i|$ for $i \in \{1,2,\ldots,N\}$, and the constant $c_1$ is  
\begin{align*}
c_1 = \frac{\alpha }{2(1-\alpha)}\kappa_C.
\end{align*}
The resulting optimal policy in vector form is 
%\begin{equation*}
%\bs{\mu} ^{\text{III}}({\bf z}_k)= -\hat{L} {\bf z}_k.
%\end{equation*}
%The matrix $\mc{L}_w = -((D\mc{L}) \otimes I_2)$, where $D$ is a diagonal matrix with
%The matrix
\begin{align*}
\bs{\mu} ^{\text{III}}({\bf z}_k)&= -\hat{L} {\bf z}_k.\\
\hat{L} &= \frac{1}{2} (R + \tilde{R})^{-1} B\tr \tilde{L}\\
&=D'B\tr \mc{L}'.
\end{align*}
In the above expression, $D' = c_1 (R + \tilde{R})^{-1}$. We can also express $D'$ as $D' = (D \otimes I_2)$  where $D$ is a diagonal matrix 
with entries
\begin{align*}
D_{ii} = \frac{1}{\frac{2(1-\alpha)}{\alpha}\frac{\kappa_M}{\kappa_C} + |\mc{N}_i|} 
\end{align*}
for all $i \in \{1,2,\ldots,N\}$. 
 \vspace{0.1in}
\begin{lem}
\emph{EAP III results in a marginally stable system dynamics}
\end{lem}
\begin{proof}
To show that the system dynamics are stable, we need to show that the eigenvalues of $I_{2(N+m)} - B\hat{L}$ lie within the unit circle and its eigenvectors are independent. Using the matrix partitioning in Eq. (\ref{eq:partitionMatrix}), the matrix $B\tr \mc{L}' \in \field{R}^{2N \times 2(N+m)}$ can be partitioned into $[\mc{L}'_f|\mc{L}'_{lf}]$ and 
\[
I_{2(N+m)} - B\hat{L}=
\left[
\begin{array}{c|c}
I_{2N} - D'\mc{L}'_f& -D' \mc{L}'_{lf}\\
\hline
0_{2m\times 2N} &  I_{2m\times 2m}
\end{array}
\right].
\]
From the properties of block matrices, $I_{2(N+m)} - B\hat{L}$ has $2m$ eigenvalues equal to one and the remaining $2N$ eigenvalues are the eigenvalues of  $I_{2N} - D'\mc{L}'_f$. The Gershgorin circles of $D'\mc{L}'_f$ are the same as those of $D\mc{L}_f$ repeated twice. For $i \in \{1,2,\ldots,N\}$, 
\begin{align*}
c_i(D\mc{L}_f)=|\mc{N}_i |D_{ii} <1~~~~ r_i(D\mc{L}_f)\leq c_i(D\mc{L}).
\end{align*}
This implies that $\lambda_i(D' \mc{L}_f') <2$ and $\lambda_i(I_{2N} - \lambda_i(D' \mc{L}_f')) <1$ for all $i$. Thus, all the eigenvalues of $I_{2(N+m)} - B\hat{L}$ lie within the unit circle. To show that all the eigenvectors are independent, $\mc{L}' = (\mc{L} \otimes I_2)$ is a laplacian matrix of an undirected graph, i.e., it is symmetric and positive semidefinite and all of its $2N$ eigenvectors are orthogonal. The remaining $2m$ eigenvectors are those of $I_{2m \times 2m}$, which are independent as well. 
%
 %$I_{2(N+m)} - B\hat{L}$ has a full set of orthogonal eigenvectors. Moreover, based on Gershgorin circle theorem, the Gershgorin circles of $D'\mc{L}'_f$ are the same as those of $D\mc{L}_f$ 
%\begin{align*}
%c_i(D\mc{L}_f)=|\mc{N}_i |D_{ii} <1~~~~ r_i(D\mc{L}_f)\leq c_i(D\mc{L}).
%\end{align*}
%Therefore, $\lambda_i(D' \mc{L}_f') <2$ and $\lambda_i(I_{2N} - \lambda_i(D' \mc{L}_f')) <1$ for all $i$ which concludes the proof.
\end{proof}

\begin{prop}\label{prop:ErrorIII}
\emph{Let the approximate cost to go }$\tilde{J}^{\text{III}}$ \emph{be as defined in Eq. (\ref{eq:ApproxCostIII}). Let} $\epsilon^{\text{III}} = \Vert \tilde{J}^{\text{III}} - J^* \Vert_{\mc{S}^*}$ \emph{and} $J^{\bs{\mu}^{\text{III}}} = \lim_{k\to \infty} T_{\bs{\mu}^{\text{III}}}^k \tilde{J}^{\text{III}}$. \emph{Then the maximum error between the global optimal solution and the approximate solution is}
\begin{equation}\label{eq:maxErrorBound}
\Vert J^{\bs{\mu}^{\text{III}}} - J^* \Vert_{\mc{S}_{\hat{\mu}}} \leq \frac{2\alpha \epsilon^{\text{III}}}{1-\alpha}
\end{equation}
\end{prop}
\begin{proof}
The proof of this proposition consists of the same steps and reasoning as the proof of Thm. \ref{prop:ErrorI}. 
\end{proof}

%%%%%%%%%%%%%%%%%%%%%%%%%%%
\section{Simulation}\label{sec:simulation}
\begin{figure}[t]
\centering
\includegraphics[scale=0.43]{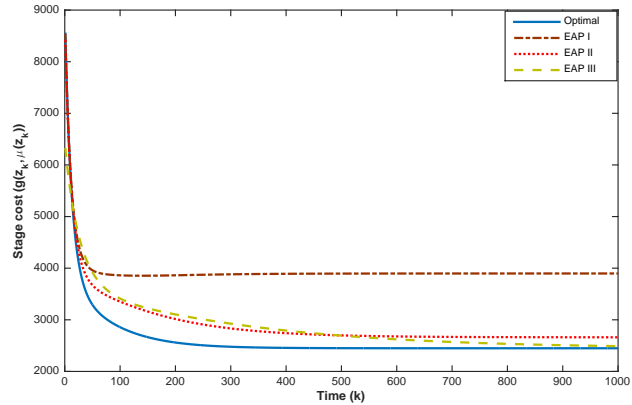}
\caption{Comparison of the stage costs of the proposed policies and the optimal policy. }
\label{fig:sim}
\end{figure}

In this section, we present the simulation results of the system under the proposed policies to verify their stability properties and to compare their performance with each other and with the global optimal policy. The details of the simulated system are as follows. The two base stations are separated by a distance $d = 100$ and are located at $[0 ~0]\tr$ and $[0 ~100]\tr$. The number of relay nodes is $N = 6$ and their initial deployment locations are ${\bf z}_0 = [{\bf x}\tr~ {\bf y}\tr]\tr$, where ${\bf x} = [5 ~12~ 26 ~30 ~33~ 38~ 0 ~100]\tr $ and ${\bf y} = [20 ~22~ 15~17~ 25~ 28~ 0~ 0]\tr$. The last two entries in ${\bf x}$ and ${\bf y}$ are the locations of the base stations. The values of the length of the communication interval, communication constant, mobility constant, and the discount factor are $T = 1000$, $\kappa_C = 1$, $\kappa_M = 1000$, and $\alpha = 0.95$ respectively. For EAPs I $\&$ II, the distributed optimization algorithm presented in Section \ref{sec:Energy-aware architectures} was implemented with number of iterations $\text{iter} = 200$ and a fixed step size $\gamma = 10^{-4}$. 

The simulation results are presented in Figs. \ref{fig:sim} and \ref{fig:trajectories}. In Fig. \ref{fig:sim}, a comparison of the costs incurred by the system under the optimal and the proposed policies is presented by plotting $g({\bf z}_k, \bs{\mu}({\bf z}_k))$ for $\bs{\mu}^*$, $\bs{\mu}^\text{I}$, $\bs{\mu}^\text{II}$, and $\bs{\mu}^\text{III}$ for all $k \in \{1,2,\ldots,T\}$. Let the total cost incurred by the system over the interval $[0,T]$ under policy $\bs{\mu}$ be $J^{\text{act}}_{\mu} = \sum_{k = 0}^T g({\bf z}_k, \bs{\mu}({\bf z}_k))$. Then for the simulated system $J^{\text{act}}_{\mu^*} = 2.62 \times 10^6$, $J^{\text{act}}_{\mu^\text{I}} =  3.94 \times 10^6$, $J^{\text{act}}_{\mu^\text{II}} = 2.91 \times 10^6$, and $J^{\text{act}}_{\mu^\text{III}} = 2.88 \times 10^6$. Based on this comparison for this particular system, the performance of EAPs II and III is close to each other and their difference from the optimal policy is small relative to EAP I. 
In Fig. \ref{fig:trajectories}, the trajectories of the relay nodes under the optimal and the proposed policies are presented. This figure provides a good insight into the performance of the proposed policies particularly EAP I. As mentioned in Sec. \ref{sec:Energy-aware architectures}, the approximate cost to go of EAP I assumes additional edges between each relay node and the origin. Consequently, the trajectories of all the relay nodes are biased towards the origin as compared to their optimal locations. This bias towards the origin plays a fundamental role in the poor performance of EAP I. A couple of interesting observations can be made by comparing Figs. \ref{subfig:Traj_EAPII} and \ref{subfig:Traj_EAPIII} with \ref{subfig:Traj_optimal}. Although the final locations of the relay nodes under EAP II are closer to the optimal terminal locations as compared to the final locations under EAP III, yet the total cost under EAP III is smaller as compared to EAP II. This observation reinforces the motivation of this work that reaching the same terminal set as the optimal solution under distributed setting may result in more energy consumption. Another interesting observation is that the relay nodes three and four first move away from their final locations and then reverse their directions. This behavior is justified because the nodes are minimizing the mobility and communication simultaneously. 

\begin{figure*}[t]
	\centering
		\subfigure[Optimal trajectories]
	{
		\includegraphics[trim = 0mm 0mm  0mm  0mm, clip, scale=0.4]{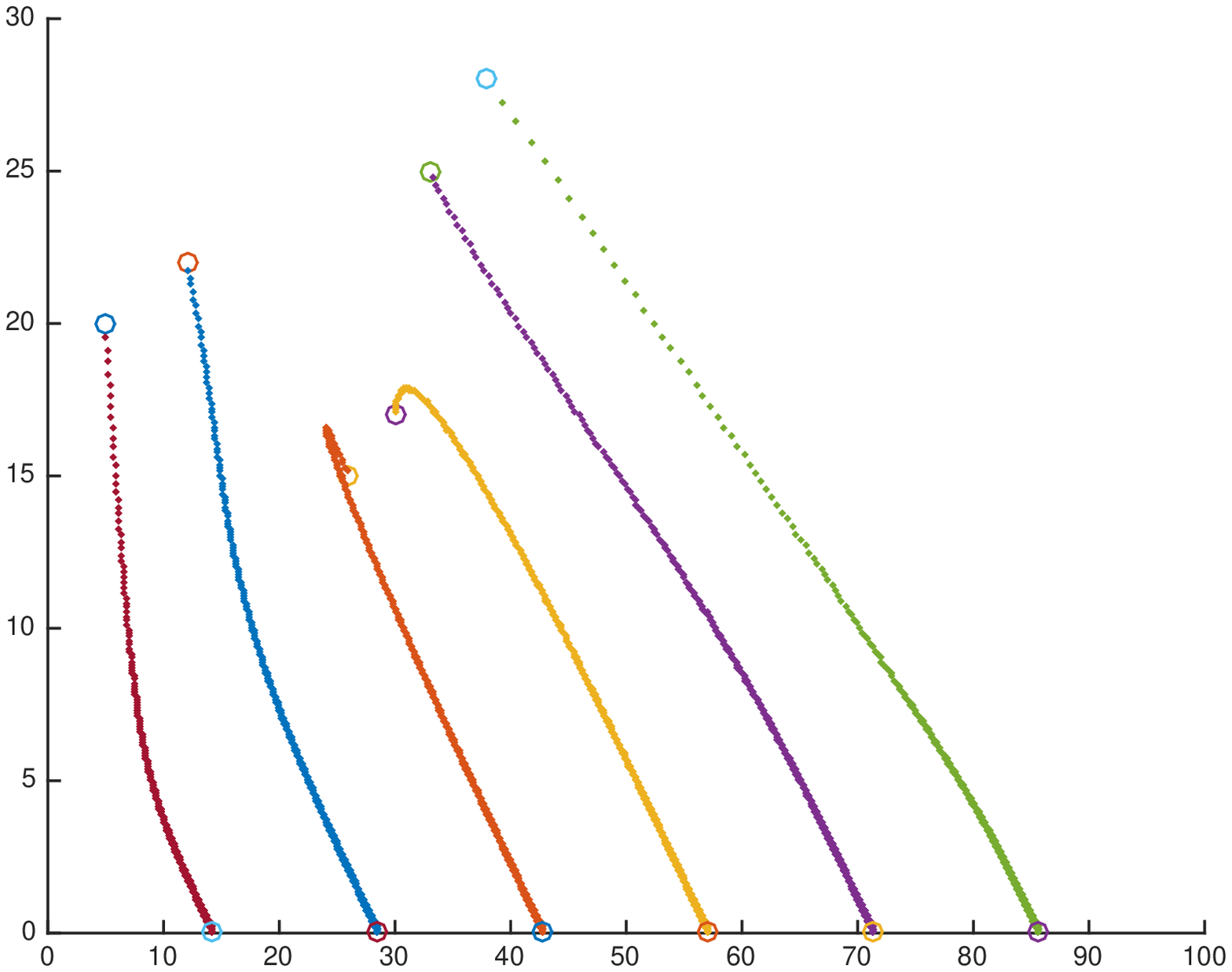}
		\label{subfig:Traj_optimal}
	}
	%\hspace{2in}
	\subfigure[Trajectories under EAP I]
	{
		\includegraphics[trim = 0mm 0mm  0mm  0mm, clip, scale=0.4]{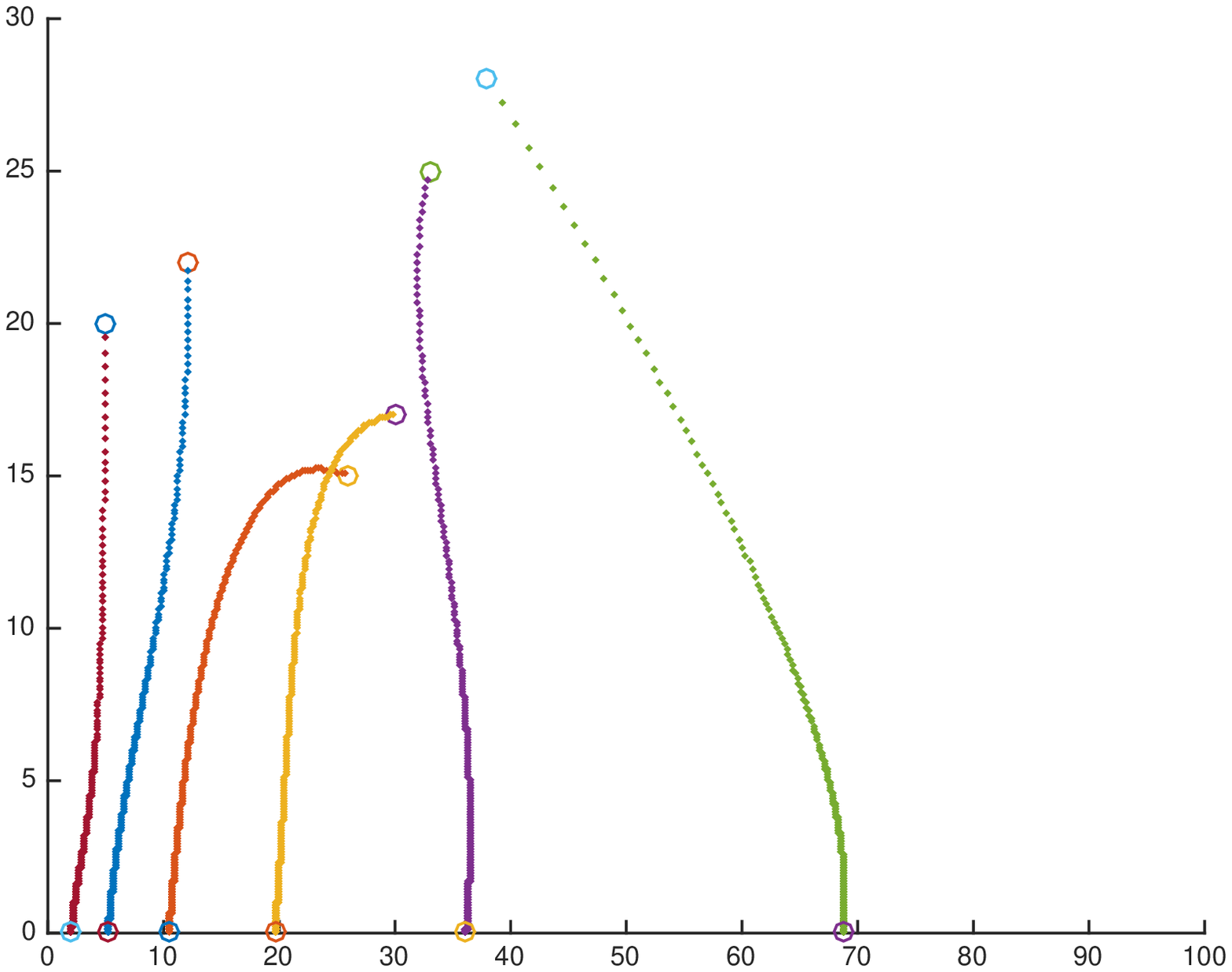}
		\label{subfig:Traj_EAPI}
	}
	\subfigure[Trajectories under EAP II]
	{
		\includegraphics[trim = 0mm 0mm  0mm  0mm, clip, scale=0.4]{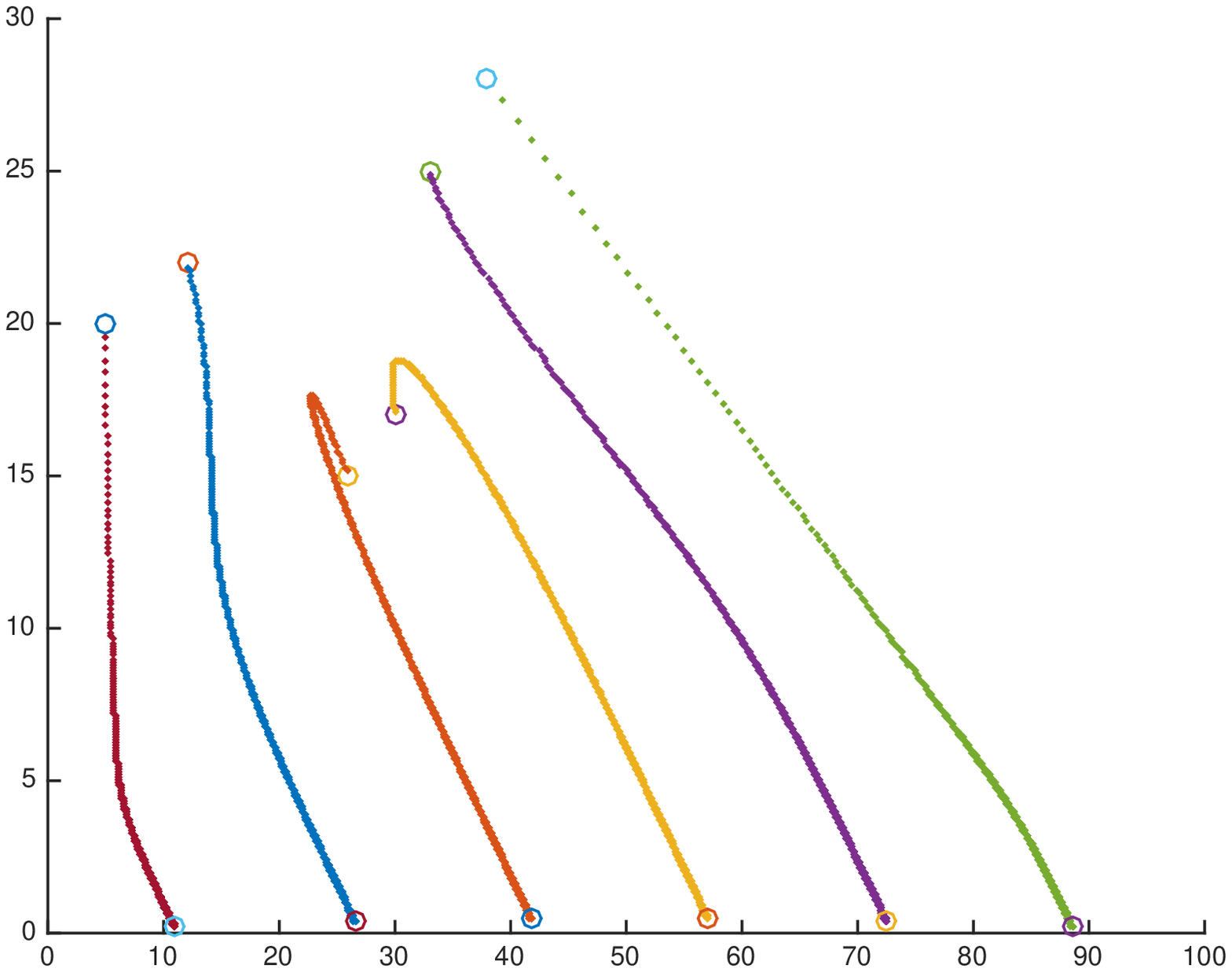}
		\label{subfig:Traj_EAPII}
	}
	\subfigure[Trajectories under EAP III]
	{
		\includegraphics[trim = 0mm 0mm  0mm  0mm, clip, scale=0.4]{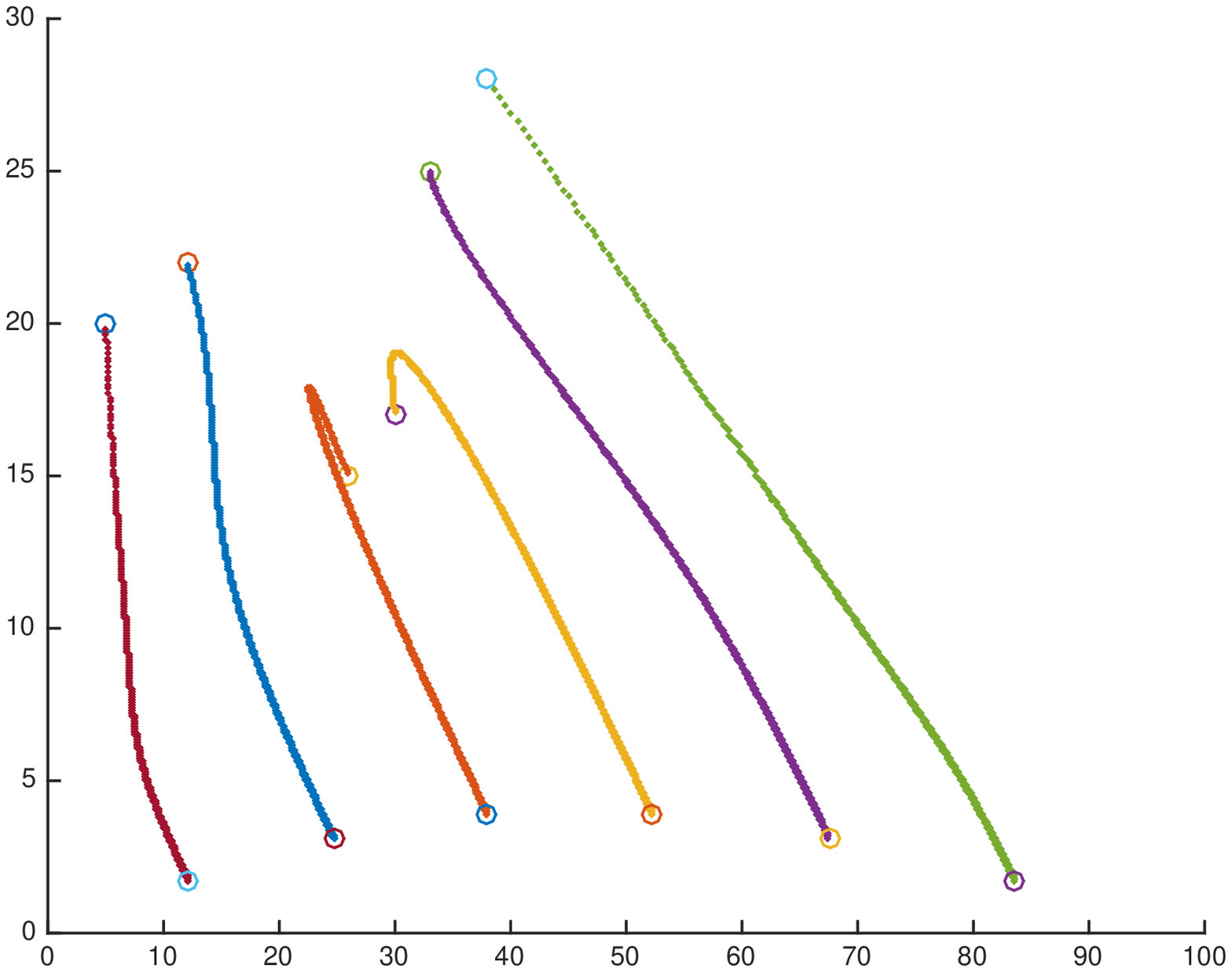}
		\label{subfig:Traj_EAPIII}
	}
	
	\caption{Trajectories of the mobile relay nodes under the optimal policy and the proposed policies EAPs I, II, $\&$ III are shown in figures (a), (b), (c), and (d) respectively. }
	\label{fig:trajectories}
\end{figure*}
\section{Conclusion}\label{sec:}
%We proposed three Energy Aware Policies for controlling the trajectories of $N$ relay nodes that need to establish an uninterrupted link between two base stations for a time interval of length $T$. The proposed schemes are designed to strike a balance between energy consumption and performance and can be implemented efficiently in real time.
We proposed an energy aware architecture for multiagent systems that can strike a balance between the global performance of the system and the cost of achieving that performance in terms of communication energy. The proposed architecture is to formulate an  infinite horizon LQR problem as an equivalent one step lookahead problem with optimal cost to go as terminal cost that can be computed offline before the deployment of the system. It was shown that to compute this cost, each agent either had to communicate with all the agents in the network or communicate extensively with its immediate neighbors to compute the optimal control action, which resulted in excessive communication overhead.

\emph{To reduce this communication overhead, the main idea behind EAPs I $\&$ II was to impose the constraints of the communication network on the global optimal cost to go and use the constrained function as an approximate cost to go}. This allowed each node to compute its control action by solving a simple parameter optimization problem using any distributed optimization algorithm.  Each node only had to exchange its current estimates of its control values and the control values of its neighbors, with its neighbors instead of exchanging estimates of all the agents as required by the globally optimal solution.

In EAP III, each node computed its control action with the assumption that all of its neighbors remain stationary. This simplifying assumption decoupled the cost of each node along the entire trajectory and allowed nodes to compute their suboptimal control action efficiently. The only information required by a node in EAP III is the current state value of its neighbors which can be either sensed if the required sensors are available or communicated. 

We analyzed the performance of the proposed schemes and computed upper bounds for the performance gap between the optimal policy and the proposed policies. We also compared the performance of the proposed schemes through simulations, which showed that EAP III performed the best for the simulated system. 

This setup provides a solid framework for energy aware algorithms for multiagent systems with focus on designing local interactions laws for individual nodes that are efficient in terms of energy consumption, can be implemented in real time on nodes that have limited energy and computation resources, and can provide minimum performance guarantees. 

\begin{appendix} \label{sec:Appendix}
\emph{ Proposition 1:} \emph{If $Q \in \field{R}^{2(N+2)\times 2(N+2)}$ is a symmetric $M$ matrix}, 
$B = \left[  \begin{array}{c}
I_{2N} \\
\hline
0_{4 \times 2N}  
\end{array} \right]
$ and 
$R = \kappa_M I_{2N}$, 
\emph{then the positive definite solution of the Riccati equation}
\begin{equation*}
K =  \alpha K -  \alpha^2 KB( \alpha B\tr K B + R)^{-1} B\tr K + Q.
\end{equation*}
\emph{is also an $M$ matrix. Moreover, if $Q$ is a laplacian matrix of a connected graph then all the entries of $K$ are non-zero. }
\\
\begin{proof}
If $B$ = $I_{2(N+2)}$ and $\alpha = 1$, then this result has been proved in Theorem 5.1 of \cite{Cao}. We will extend this result for $B$ and $\alpha$ as defined above. 
We start the proof by representing the matrices $Q$ and $K$ as block matrices following the convention introduced in (\ref{eq:partitionMatrix}): \\
$
~~~~~~~~~~~~~~~~~~~~~~~~~~~~~~~~~~~~~~~~~~~~Q = \left[ \begin{array}{c|c}
Q_f & Q_{fl} \\
\hline 
Q_{fl}\tr&Q_l
\end{array}
\right]
$
and 
$
K= \left[ \begin{array}{c|c}
K_f & K_{fl} \\
\hline 
K_{fl}\tr&K_l
\end{array}
\right],
$
\\
Then,
\begin{align}\label{eq:RinKf}
Q &= (1-\alpha)K + \alpha^2 KB( \alpha B\tr K B + R)^{-1} B\tr K.\nonumber\\
B\tr Q B &= (1-\alpha)B\tr K B + \alpha^2 B\tr K B (I + \alpha B\tr K B R^{-1})^{-1} R^{-1} B\tr K B\nonumber\\
R^{-1}  Q_f &= (1-\alpha) R^{-1} K_f + \alpha^2 R^{-1}  K_f (I + \alpha K_f R^{-1} )^{-1} R^{-1} K_f
\end{align}
The last equality is based on matrix decomposition, $B\tr K B = K_f$ and $B\tr Q B = Q_f$ and multiplying both sides by $R^{-1} $. Because $R = \kappa_M I_{2N}$, it commutes with any square matrix. Using matrix inversion lemma 
\begin{align*}
R^{-1} Q_f &= R^{-1} K_f - \alpha(I + \alpha K_f R^{-1} )^{-1} R^{-1}  K_f\\
(R^{-1} K_f)(R^{-1} Q_f) &= (R^{-1} K_f)^2 - \alpha R^{-1} K_f(I + \alpha K_f R^{-1} )^{-1}  R^{-1}  K_f
\end{align*}
From Eq. (\ref{eq:RinKf}), $ \alpha^2 R^{-1}  K_f (I + \alpha K_f R^{-1} )^{-1}R^{-1} K_f  =R^{-1}  Q_f -  (1-\alpha) R^{-1} K_f$. Using this equality in the above equation yields the following quadratic equation.
\begin{align*}
(R^{-1} K_f)^2 - R^{-1}K_f\left(R^{-1}Q_f - \frac{1-\alpha}{\alpha} I \right) - \frac{1}{\alpha} R^{-1}Q_f = 0
\end{align*}
After performing a series of simple algebraic manipulations
\begin{equation*}
R^{-1}K_f = G + \sqrt{G + \frac{4}{\alpha}}\sqrt{G} - \frac{2(1-\sqrt{\alpha})}{\alpha} I_{2N},
\end{equation*}
where $G = R^{-1}Q_f + \frac{(1-\sqrt{\alpha})^2}{\alpha}$. Since $Q_f$ is a symmetric positive semidefinite $M$ matrix, $G$ is a positive definite $M$ matrix and $G_s = \sqrt{G}$ is also a positive definite $M$ matrix \cite{Alefeld}, and all of its entries are non-zero \cite{Cao}. 
Using \emph{Lemma 5.5} of \cite{Cao}, $G + \sqrt{G + \frac{4}{\alpha}}\sqrt{G}$ is an $M$ matrix with negative off-diagonal entries. Now, $R^{-1}K$ is a positive definite matrix, so its principal submatrix $R^{-1}K_f$ is also positive definite. Thus, all the diagonal entries of $R^{-1}K_f$ are positive proving that $K_f$ is a positive definite $M$ matrix. 

Next we will show that all the entries of $K_{fl}$ are negative. Using the definition of $B$, 
\begin{align*}
Q_{fl} &= (1-\alpha) K_{fl} + \alpha^2 K_f (\alpha K_f + R)^{-1}K_{fl} \\
K_{fl} &= [(1-\alpha)I + \alpha^2 K_f (\alpha K_f + R)^{-1}]^{-1}Q_{fl} \\
&= (1-\alpha)\left(sI - D\right)^{-1}Q_{fl}
\end{align*}
where $s = \left(1+ \frac{\alpha^2}{1-\alpha}\right)$ and $D = \frac{\alpha^2}{1-\alpha}(I+\alpha R^{-1}K_f)^{-1}$. Because $\alpha R^{-1}K_f$ is an $M$ matrix, $D$ is a positive matrix. Furthermore, $\rho((I+\alpha R^{-1}K_f)^{-1}) <1$, which implies that $\rho(D) < s$ and $(1-\alpha)\left(sI - D\right)^{-1}$ is a positive matrix. Since $Q_{fl}$ has all the entries non-positive, $K_{fl}$ has all the entries strictly negative unless an entire column of $Q_{fl}$ is zero which is not possible for a connected graph.

Finally, to show that $K_l$, 
\begin{align*}
Q_l = (1-\alpha) K_l + \alpha^2 K_{fl}\tr (\alpha K_f + R)^{-1}K_{fl} 
\end{align*}
Since $(\alpha K_f + R)^{-1}$ is a positive matrix and all the entries of $K_{fl}$ are negative, $\alpha^2 K_{fl}\tr (\alpha K_f + R)^{-1}K_{fl}$ is a positive matrix. By definition $Q_l$ has positive diagonal entries and non-positive off-diagonal entries. To satisfy the above equation, $K_l$ must have positive diagonal entries and negative off-diagonal entries. This concludes the proof. 
\end{proof}
\end{appendix}

\end{document}